\documentclass[a4paper,twoside]{article}

\usepackage{a4}
\usepackage[left=2.5cm,right=2.0cm,top=1.5cm,bottom=1cm,includeheadfoot]{geometry} 

\usepackage[ngerman, english]{babel}

\usepackage{lmodern}
\usepackage[T1]{fontenc}
\usepackage{textcomp}

\usepackage[utf8]{inputenc}

\usepackage{autobreak}
\usepackage{amsfonts}
\usepackage{amstext}
\usepackage{amssymb}
\usepackage{amsmath}
\usepackage{amsthm}
\usepackage{mathrsfs} 
\usepackage{wasysym}
\usepackage{mathtools}
\usepackage{appendix}

\usepackage{graphicx}
\usepackage{tikz}
\usepackage{tikz-cd}
\usepackage{tikz-3dplot}
\usetikzlibrary{matrix, calc, arrows,intersections,shapes.arrows}
\usepackage{pgfplots}
\pgfplotsset{
  compat=1.12,
  axis lines=middle
}

\usepackage{listings}
\usepackage{xcolor}

\usepackage{pdfpages}
\usepackage[pdfborder={0 0 0},colorlinks=false]{hyperref}
\usepackage[capitalize,nameinlink,noabbrev]{cleveref}

\usepackage[babel,german=quotes]{csquotes}
\usepackage[backend=biber, style=numeric]{biblatex}
\renewbibmacro*{in:}{}

\DeclareFieldFormat[article]{title}{\mkbibemph{#1}}
\DeclareFieldFormat[article]{journaltitle}{#1}
\DeclareFieldFormat[article]{pages}{#1}

\AtEveryBibitem{%
  \clearfield{year}%
  \clearfield{date}%
}

\renewbibmacro*{date}{}

\renewbibmacro*{issue+date}{}

\DeclareFieldFormat[article]{volume}{\mkbibbold{#1}}
\DeclareFieldFormat[article]{number}{\mkbibbold{#1}}

\makeatletter
\newtheoremstyle{theorem}
  {}
  {}
  {\itshape}
  {}
  {\bfseries}
  {.}
  { }
  {\thm@notefont{\bfseries}\thmhead@plain{#1}{#2}{#3}}
  
\newtheoremstyle{definition}
  {}
  {}
  {}
  {}
  {\bfseries}
  {.}
  { }
  {\thm@notefont{\bfseries}\thmhead@plain{#1}{#2}{#3}}  
\makeatother

\theoremstyle{definition}
\newtheorem{definition}{Definition}[section]
\crefname{definition}{Definition}{Definitions}

\crefname{assumption}{Assumption}{Assumptions}

\theoremstyle{theorem}
\newtheorem{theorem}[definition]{Theorem}
\crefname{theorem}{Theorem}{Theorems}
\newtheorem{theoremx}{Theorem}
 
\newtheorem{prop}[definition]{Proposition}
\crefname{prop}{Proposition}{Propositions}
\newtheorem{lemma}[definition]{Lemma}
\crefname{lemma}{Lemma}{Lemmas}
\newtheorem{corollary}[definition]{Corollary}
\crefname{corollary}{Corollary}{Corollaries}

\usepackage{enumitem}

\newcommand{\comment}[1]{}

\newcommand{\IN}{\ensuremath{\mathbb{N}}}
\newcommand{\IZ}{\ensuremath{\mathbb{Z}}}
\newcommand{\IQ}{\ensuremath{\mathbb{Q}}}
\newcommand{\IR}{\ensuremath{\mathbb{R}}}
\newcommand{\IC}{\ensuremath{\mathbb{C}}}
\newcommand{\IS}{\ensuremath{\mathbb{S}}}

\DeclareMathOperator{\Hom}{Hom}

\DeclareMathOperator{\Mat}{Mat}
\DeclareMathOperator{\GL}{GL}
\DeclareMathOperator{\coker}{coker}

 {}  
 {}  

\DeclareMathOperator{\id}{id} 		
\DeclareMathOperator{\im}{im} 		

\DeclareMathOperator{\Hess}{Hess} 			
\DeclareMathOperator{\ind}{ind} 			

\let\originalleft\left
\let\originalright\right
\renewcommand{\left}{\mathopen{}\mathclose\bgroup\originalleft}
\renewcommand{\right}{\aftergroup\egroup\originalright}

\begin{document}

\title{Isotropy sets of periodic geodesic flows on \\three-dimensional manifolds}
\author{Jakob Dittmer}
\date{}

\maketitle



\begin{abstract}
\noindent

A classical theorem of Wadsley states that if all geodesics on a Riemannian manifold $M$ are periodic, then they have a common period, and the geodesic flow induces an $\IS^1$-action on the unit tangent bundle. Moreover, a conjecture of Berger asserts that this action is free whenever $M$ is simply connected.\\
In this paper, we investigate the Berger conjecture in dimension three by studying the isotropy set of the $\IS^1$-action induced by the geodesic flow. Using Morse theory for the energy functional on the free loop space of $M$, we establish two results in this direction. First, we show that all isotropy groups of the geodesic flow have odd order. Second, we prove that the geodesic flow either acts freely on the unit tangent bundle $T^1M$ or has at least two singular orbit types.\\
This paper serves as preparation for the subsequent paper The Berger conjecture for three-dimensional manifolds, in which the conjecture is proved in dimension three.
 
\end{abstract}

\tableofcontents

\section{Introduction}

Given a Riemannian manifold $(M,\langle \cdot,\cdot\rangle )$, a geodesic $\gamma:I\rightarrow M$ defined on some interval $I\subset \IR$ is called a closed geodesic if it extends to a periodic map $\gamma:\IR \rightarrow M$. If $\gamma$ is unit speed, then its shortest period is called the length of $\gamma$. Manifolds all of whose geodesics are closed have been objects of special interest for quite a while, mainly since Zoll and Tannery discovered metrics with this property on the sphere that are different from the round metric. The most common examples of manifolds all of whose geodesic are closed are compact symmetric Spaces of rank $1$, the so-called CROSSes, namely the spheres equipped with the round metric, the real, complex and quaternionic projective spaces equipped with their Fubini-Study metric and the Cayley plane. A theorem of Wadsley says that if $(M,\langle \cdot, \cdot\rangle)$ is a manifold all of whose geodesics are closed, then all geodesics have a common period. This implies that the geodesic flow induces a smooth $\IS^1$-action on the unit tangent bundle $T^1M$ of $M$. 
A conjecture of Berger states that on a simply connected manifold with this property all unit speed geodesics have the same minimal period. This means that the least common period that exists by Wadsley's theorem is also the smallest period of each single unit speed geodesic. This is the case if and only if the $\IS^1$-action of the geodesic flow on the unit tangent bundle is free. 
The conjecture was proven to be true for two-dimensional manifolds by Gromoll and Grove in \cite{gromoll1981metrics}. Marco Radeschi and Burkhard Wilking were the first to use Morse theory on the free loop space to attack the conjecture. In \cite{radeschi2017berger}, they showed that the conjecture also holds for topological spheres of dimension greater than or equal to four. In dimension three, the very same approach does not directly lead to a proof of the full Berger conjecture, but John Olsen was able to use it to show that the isotropy set of the geodesic flow consists of a disjoint union of finitely many three-dimensional connected manifolds with fixed isotropy group, each of which is diffeomorphic to a lens space whose fundamental group has even order. In this paper we will sharpen John Olsen's results about the isotropy set of the geodesic flow. First we show that the order of the isotropy group of each component of the isotropy set of the geodesic flow is odd.  

\begin{theoremx}\label{isoorder}
Let $(M,\langle \cdot,\cdot\rangle)$ be a simply connected three dimensional Riemannian manifold with all geodesics closed. Then all the isotropy groups of the geodesic flow have odd order.
\end{theoremx}

Apart from this we have the following result about the number of connected components of the isotropy set of the geodesic flow.

\begin{theoremx} \label{thmoneorbittype}
Let $(M,\langle \cdot,\cdot \rangle)$ be a three dimensional Riemannian manifold all of whose geodesics are closed. Then the $\IS^1$-action induced by the geodesic flow is either free or its isotropy set has at least two connected components.
\end{theoremx} 

\comment{
We study the topology of the free loop space 
\begin{align*}
\Lambda M\coloneqq \mathcal{H}^1(\IS^1,M),
\end{align*}
namely the Hilbert manifold of Sobolev $H^1$-maps from $\IS^1$ to $M$. More precisely, we consider the energy functional $E:\Lambda M\rightarrow \IR$ given by the assignment
\begin{align*}
E(c)=\int_0 ^1 |\dot{c}(t)|^2dt.
\end{align*}}

In order to apply Morse theory of the energy functional we relate subset of isotropy set with the same isotropy type of the geodesic flow to critical set of the Energy functional with a fixed critical energy. A critical set of a given energy will be a connected manifold and either diffeomorphic to a component of the isotropy set of the geodesic flow or to $T^1M$. We will show that the Energy functional is perfect with respect ordinary rational cohomology and use this to deduce that the index of a closed geodesic is monotone with respect to the energy. Afterwards we will compute the $\IS^1$-equivariant  and ordinary cohomology groups of the critical sets.\\

To show that the isotropy groups are all odd, we consider the groups $H^{*} _{\IS^1}(\Lambda;\IZ/2\IZ)$. We will derive a recursion formula for the Betti numbers of $\Lambda M$ and count the number of critical manifolds that have smaller energy than the smallest energy of an critical manifold that is diffeomorphic to $T^1M$. We will show that this number is even.\\

For the second result we note that if the isotropy set of the geodesic flow is connected, the closed geodesics with the second lowest critical energy will be obtained by iterating the geodesics of the lowest energy. Moreover, the critical manifolds of the lowest and second lowest energy will be singular. We will use Morse theory and the cohomology of the critical sets to compute the groups $H^{*} _{\IS^1}(\Lambda M,M ;\IZ)$ and $H^{*}(\Lambda M,M;\IZ)$ for $*\leq 10$. It turns out that the result, we obtain under this assumption differs from the result, we would obtain, if we assumed that $\langle \cdot,\cdot\rangle $ is the round metric and the geodesic flow acts freely.\\

{\bf Acknowledgements}: This paper is part of the author's doctoral thesis at the University of Münster. The author would like to express his deep gratitude to his advisor, Burkhard Wilking, for giving him the opportunity to work on this project, for his guidance and encouragement, and especially for suggesting how to attack the Berger conjecture. He would also like to warmly thank his mentor, Johannes Ebert, as well as Peter Albers and Michael Wiemeler, for many helpful conversations and valuable advice.


\newpage


\section{Preliminaries}

\subsection{Geodesic flow}\label{sec::critmanhom}

Let $(M,\langle\cdot,\cdot\rangle)$ be a simply connected three-dimensional Riemannian manifold all of whose geodesics are closed. We always parametrize geodesics by arc length unless stated otherwise. Throughout this paper we denote the canonical bundle projection of the tangent bundle $TM$ to $M$ by $\pi_M$. When all geodesics are closed, Wadsley's theorem shows that there is a least common period $L>0$ for all unit speed geodesics. Hence the geodesic flow induces an $\IS^1\cong \IR/L\IZ$-action on the unit tangent bundle $T^1M$. The Berger conjecture is equivalent to this action being free.\\

\noindent Since $M$ is three-dimensional, $M$ must be diffeomorphic to $\IS^3$ and $T^1M\cong \IS^3\times\IS^2$. By \cite{olsen2010three}, in this case the isotropy set of the geodesic flow is a finite disjoint union \(
C_1\sqcup\cdots\sqcup C_k,\) where each $C_j$ has isotropy group $\IZ/q_j\IZ$ for some $q_j\in \IZ$ and is diffeomorphic to a lens space with $\#\pi_1(C_j)$ even.

\subsection{Free loop space}
The \emph{free loop space} of $M$ is defined to be the Sobolev space 
$$
\Lambda M \coloneqq W^{1,2}(\IS^1,M).
$$
Its elements are the continuous maps $c:S^1\rightarrow M$ whose derivative $\dot{c}(t)$ exist almost everywhere and are square integrable with respect to the Riemannian metric. It is a smooth Hilbert manifold and carries a continuous $\IS^1$-action given by reparametrization of the loops, namely
\[
([t],c)\longmapsto (s\mapsto c(s+t)).
\]
It comes with a canonical fibration $p:\Lambda M\rightarrow M, \gamma \mapsto \gamma(0)$, given by evaluating a curve at zero. 
For more details on the free loop space we refer the interested reader to the book of Klingenberg (\cite{klingenberg2012lectures}). Since $M$ is in our case diffeomorphic to $\IS^3$, the free loop space $\Lambda M$ is homeomorphic to the free loop space of $\IS^3$. This means in particular that we obtain for any coefficient ring $R$ isomorphisms
\begin{align}
H^*(\Lambda M;R)\cong H^{*} (\Lambda \IS^3;R)\cong H^*(\Omega\IS^3 ;R)\otimes H^*(\IS^3;R),
\end{align}
where $\Omega\IS^3$ is the loop space over some point in $\IS^3$. 

\subsection{Critical manifolds of the energy functional}
The energy functional is defined by
\[
E:\Lambda M\longrightarrow\IR,\qquad E(c)=\int_{\IS^1}|\dot c(t)|^2\,dt.
\]
It is smooth and $\IS^1$-invariant, and its critical points are precisely the closed geodesics (See \cite{klingenberg2012lectures} for more details). If all geodesics are closed, then $E$ is a Morse--Bott function; thus every connected component of the critical set of $E$ is a non-degenerate critical submanifold of $\Lambda M$. Having this in mind we will throughout this paper refer to these connected components as \emph{critical manifolds}. We write
\[
0=e_0<e_1<e_2<\cdots
\]
for the critical values of $E$, set
\[
K_i:=\operatorname{Crit}(E)\cap E^{-1}(e_i),
\]
and choose $\epsilon_i>0$ so that the intervals $(e_i-\epsilon_i,e_i+\epsilon_i)$ are pairwise disjoint. We denote the corresponding sublevel sets by
\[
\Lambda^i:=E^{-1}(( -\infty,e_i+\epsilon_i]).
\]
The critical set $K_0$ consists of the constant loops, and can therefore be identified with $M$.\\

\noindent If $K$ is a component of a critical set with positive energy, then evaluation at $0$ and normalization of the velocity defines an embedding
\[
\Phi:K\longrightarrow T^1M,\qquad
c\longmapsto\left(c(0),\frac{\dot c(0)}{\sqrt{E(c)}}\right).
\]
Its image is either all of $T^1M$ or one of the isotropy components $C_j$. Accordingly, we call $K$ \emph{regular} or \emph{singular}. We seem that a closed geodesic is regular if its length is a multiple of $L$, and singular otherwise. If $K$ is singular, then $K$ is diffeomorphic to a lens space and
\begin{align}\label{eqrelativehomology}
H^*(K;\IZ)\cong
\begin{cases}
\IZ,& *=0,3,\\
\IZ/(\#\pi_1(K))\IZ,& *=2,\\
0,&\text{otherwise},
\end{cases}
\end{align}
with $\#\pi_1(K)$ even. We write $c^q(t):=c(qt)$ for the $q$-fold iterate of a closed geodesic $c$ and  for each critical manifold $K\subset \Lambda M$ we denote by $K^q$ the critical manifold, consisting of the $q$-times iterations of $K$. For each isotropy component $C_j$, let $\widetilde K_j$ denote the component of the critical set of least energy with $\Phi(\widetilde K_j)=C_j$, it consists of prime geodesics, and all further connected components of critical sets mapping to $C_j$ are obtained by iteration.  \\
Finally, $L^2$ is the smallest critical value whose critical set is regular. We denote by $i_0$ the minimal index such that $K_{i_0}$ is regular. For this critical set,
\[
\Phi:K_{i_0}\longrightarrow T^1M,\qquad
c\longmapsto\left(c(0),\frac{\dot c(0)}{L}\right)
\]
is an $\IS^1$-equivariant diffeomorphism.
\subsection{Morse--Bott theory and the topology of the free loop space}\label{sec::topologyfreeloopspace}

Let $B\subset K_i$ be a connected component of the critical set. Since $E$ is Morse--Bott, the Hessian gives for every $c\in B$ an orthogonal splitting
\begin{align}\label{eq:tangentspacedecomp}
T_c\Lambda M=T_cB\oplus T_c\Lambda M^-\oplus T_c\Lambda M^+,
\end{align}
where the Hessian $\Hess E$ is zero on $T_cB$, the Hession $\Hess E$ is negative definite on $T_c\Lambda M^-$ and positive definite on $T_c\Lambda M^+$. 
The space $T_c\Lambda M^+$ has finite dimension and its dimension is called the \emph{index} of $c$. Since each critical manifold is connected, the index is constant along $B$.
It follows that the spaces $T_c\Lambda M^+$ form a finite-dimensional vector bundle $N(B)\to B$, called the negative bundle. By \cite{olsen2010three,radeschi2017berger}, these bundles $N(B)$ is orientable, if $M$ is a sphere. Moreover, the $\IS^1$-action preserves its orientation.\\

\noindent The gradient flow gives a diffeomorphism between sublevel sets corresponding to two values lying between consecutive critical values. Furthermore, after choosing the numbers $\varepsilon_i$ sufficiently small, Morse--Bott theory implies (see \cite{meyer1967kritische} for more details), that up to equivariant homotopy $\Lambda^i$ is obtained from $\Lambda^{i-1}$ by attaching for any critical manifold $B$ with energy $e_i$ handlebodies of type $DN(B)\oplus D\nu^{+}(B)$  along the boundary of $DN(B)$. In particular there are isomorphisms 
\begin{align}\label{sublevelsetsrealsublevelsetsiso}
H^*(\Lambda^i,\Lambda^{i-1};R)
&\cong \bigoplus_{B\subset K_i}H^*(DN(B),\partial DN(B);R),\\
\label{sublevelsetsrealsublevelsetsisoequiv}
H_{\IS^1}^*(\Lambda^i,\Lambda^{i-1};R)
&\cong \bigoplus_{B\subset K_i}H_{\IS^1}^*(DN(B),\partial DN(B);R).
\end{align}
Since $N(B)$ is orientable, the Thom isomorphism gives
\begin{align}\label{sublevelthomiso}
H^*(DN(B),\partial DN(B);R)&\cong H^{*-\ind(B)}(B;R),\\
\label{subsevelthomisoequivariant}
H_{\IS^1}^*(DN(B),\partial DN(B);R)&\cong H_{\IS^1}^{*-\ind(B)}(B;R).
\end{align}
Consequently,
\begin{align}\label{sublevelfinalthomiso}
H^*(\Lambda^i,\Lambda^{i-1};R)
&\cong \bigoplus_{B\subset K_i}H^{*-\ind(B)}(B;R),\\
\label{sublevelfinalthomisoequiv}
H_{\IS^1}^*(\Lambda^i,\Lambda^{i-1};R)
&\cong \bigoplus_{B\subset K_i}H_{\IS^1}^{*-\ind(B)}(B;R).
\end{align}
These two formulas are the Morse--Bott input used below to reconstruct the ordinary and equivariant cohomology of the free loop space from its critical manifolds. For a more detailed treatment we refer the interested reader to the work of Wolfgang Meyer \cite{meyer1967kritische}.


\section{Monotonicity of the index}

Using the preceding section, we now deduce two basic properties of the energy functional on three-dimensional manifolds: rational perfectness and monotonicity of the indices of its critical manifolds.

\begin{lemma}\label{lemrationalperfectness}
The Morse--Bott function $E$ is perfect with respect to the rational cohomology of $\Lambda M$.
\end{lemma}

\begin{proof}
Each critical manifold is either a $\IQ$-cohomology sphere or diffeomorphic to $T^1M\cong \IS^3\times \IS^2$. By \cite{wilking2001index}, the index of every critical manifold has parity opposite to $\dim M$; since $\dim M=3$, all indices are even. Moreover, if $c^q$ is a regular geodesic and $0<l<q$, the index gap formula (Proposition 4.2 in \cite{radeschi2017berger}) gives
\begin{align}
\ind(c^{q+l})&=\ind(c^q)+\ind(c^l)+2, \label{indexformula1}\\
\ind(c^q)&\geq \ind(c^l)+2. \label{indexformula2}
\end{align}
Hence, if $c$ is regular, there is no geodesic $\tilde c$ with $\ind(\tilde c)=\ind(c)+2$; in particular, if $K$ is regular, there is no critical manifold of index $\ind(K)+2$.\\
By Corollary 3.1 in \cite{olsen2010three}, there is also at most one critical manifold of any given index. Consequently,
\begin{align*}
\bigoplus_{i\in\IN}H^*(\Lambda^i,\Lambda^{i-1};\IQ)
&\cong \bigoplus_{\substack{B\ \textup{critical manifold}\\ \textup{of energy}}}H^{*-\ind(B)}(B;\IQ)\\
&\cong
\begin{cases}
\IQ,& *=0\ \text{or}\ *\geq2,\\
0,& *=1.
\end{cases}
\end{align*}
These are the rational cohomology groups of $\Lambda\IS^3$. The Morse inequalities therefore force all boundary maps in the corresponding long exact sequences to vanish, proving rational perfectness.
\end{proof}

\begin{prop}[Index monotonicity]\label{propindexmonotonicity}
As in Proposition \ref{propmappingdegreecorespondence}, let $B_1$ and $B_2$ be critical manifolds of critical energies $e_i$ and $e_j$. If $e_i<e_j$, then $\ind(B_1)<\ind(B_2)$.
\end{prop}

\begin{proof}
Assume, to the contrary, that $i<j$ but $\ind(B_1)>\ind(B_2)$. Rational perfectness gives
\begin{align}\label{eqrationalperfectnessusage1}
H^*(\Lambda^i;\IQ)
&\cong \bigoplus_{l\leq i}H^*(\Lambda^l,\Lambda^{l-1};\IQ)
\cong \bigoplus_{\substack{B\ \textup{critical manifold}\\ \textup{of energy}\ \leq e_i}}H^{*-\ind(B)}(B;\IQ).
\end{align}
Thus the inclusion $\iota:\Lambda^i\hookrightarrow\Lambda M$ induces isomorphisms in degrees $\ind(B_1)$ and $\ind(B_1)+3$. Since $B_1$ is the only critical manifold of index $\ind(B_1)$ and $e_i<e_j$, we also have
\[
H^{\ind(B_2)}(\Lambda^i;\IQ)=0.
\]

For $p\in\IS^3$, the ring $H^*(\Omega_pM;\IQ)$ is polynomial on one generator of degree $2$, and
\[
H^*(\Lambda M;\IQ)\cong H^*(\Omega_p\IS^3;\IQ)\otimes_{\IQ}H^*(\IS^3;\IQ).
\]
Hence there is $c\in H^2(\Lambda M;\IQ)$ such that cup product with $c^l$ induces an isomorphism
\[
\cup c^l:H^{2*}(\Lambda M;\IQ)\xrightarrow{\cong}H^{2*+2l}(\Lambda M;\IQ)
\qquad(l\in\IN).
\]
Naturality of the cup product gives the commutative diagram
\begin{center}
\begin{tikzpicture}[commutative diagrams/every diagram]
\matrix[matrix of math nodes, name=m,column sep=3cm,row sep=1.1cm,commutative diagrams/every cell] {
H^{\ind(B_2)}(\Lambda M;\IQ)&H^{\ind(B_1)}(\Lambda M;\IQ)\\
H^{\ind(B_2)}(\Lambda^i;\IQ)\cong0&H^{\ind(B_1)}(\Lambda^i;\IQ)\\};
\path[commutative diagrams/.cd,every arrow,every label]
(m-1-1) edge node {$\cup c^{\frac{\ind(B_1)-\ind(B_2)}2}$}(m-1-2)
(m-2-1) edge node {$\cup\iota^*\!\left(c^{\frac{\ind(B_1)-\ind(B_2)}2}\right)$}(m-2-2)
(m-1-1) edge node {$\iota^*$}(m-2-1)
(m-1-2) edge node {$\iota^*$}(m-2-2);
\end{tikzpicture}
\end{center}
The upper horizontal map is an isomorphism, as is the right vertical map by rational perfectness, whereas the lower-left group vanishes. This contradiction proves the claim.
\end{proof}

\begin{corollary}\label{corconnectedcritsets}
For each $i\in\IN$, the critical set $K_i$ of energy $e_i$ is connected. The orders of the isotropy groups of $C_1,\ldots,C_k$ are pairwise distinct.
\end{corollary}

\begin{proof}
Suppose that $B$ and $B'$ are distinct critical manifolds of the same energy. By Proposition \ref{propindexmonotonicity}, they must have the same index, whereas Corollary 3.1 in \cite{olsen2010three} allows at most one critical manifold of a fixed index. Hence the result follows.
\end{proof}

\begin{corollary}\label{totalcohomiso}
Let $i\in\IN$ and let $R$ be a coefficient ring. For $*\leq\ind(K_{i+1})-2$, the inclusion $\Lambda^i\hookrightarrow\Lambda M$ induces isomorphisms
\begin{align*}
H^*_{\IS^1}(\Lambda M,M;R)&\xrightarrow{\cong}H^*_{\IS^1}(\Lambda^i,M;R),\\
H^*_{\IS^1}(\Lambda M;R)&\xrightarrow{\cong}H^*_{\IS^1}(\Lambda^i;R).
\end{align*}
\end{corollary}

\begin{proof}
By index monotonicity and equations (\ref{sublevelfinalthomisoequiv}) and (\ref{sublevelfinalthomiso}),
\[
H^*_{\IS^1}(\Lambda^j,\Lambda^{j-1};R)=0,
\qquad
H^*(\Lambda^j,\Lambda^{j-1};R)=0
\]
for $*\leq\ind(K_{i+1})-1\leq\ind(K_j)-1$. The pair sequence of $(\Lambda^j,\Lambda^{j-1})$ and the triple sequence of $(\Lambda^j,\Lambda^{j-1},M)$ therefore give, for $j\geq i$,
\[
H^*_{\IS^1}(\Lambda^j,M;R)\cong H^*_{\IS^1}(\Lambda^i,M;R),
\qquad
H^*(\Lambda^j,M;R)\cong H^*(\Lambda^i,M;R).
\]
Passing to the inverse limit yields
\begin{align*}
H^*_{\IS^1}(\Lambda M,M;R)
&\cong \lim_{j\to\infty}H^*_{\IS^1}(\Lambda^j,M;R)
\cong H^*_{\IS^1}(\Lambda^i,M;R),\\
H^*(\Lambda M;R)
&\cong \lim_{j\to\infty}H^*(\Lambda^j;R)
\cong H^*(\Lambda^i;R).
\end{align*}
\end{proof}

Recall from Section 1.1 that $i_0$ is minimal such that $K_{i_0}$ contains regular geodesics. Then $K_1,\ldots,K_{i_0-1}$ are three-dimensional and have the integral cohomology ring of a lens space with even fundamental group. Rational perfectness and index monotonicity imply $\ind(K_i)=2i$ for these critical manifolds. Together with (\ref{eqrelativehomology}), this gives, for $i\leq i_0$,
\begin{align}\label{eqrelativehomology2}
H^*(\Lambda^i,\Lambda^{i-1};\IZ)\cong
\begin{cases}
\IZ,& *=2i,2i+3,\\
\IZ/\#\pi_1(K_i)\IZ,& *=2i+2,\\
0,&\textup{elsewhere}.
\end{cases}
\end{align}
This yields the following basic consequence.

\begin{corollary}\label{propbasicisos}
Suppose $0<i<i_0$. Then the inclusion
\[
\iota_i:\Lambda^i\hookrightarrow\Lambda M
\]
induces isomorphisms
\[
\iota_i^*:H^*(\Lambda M;\IZ)\xrightarrow{\cong}H^*(\Lambda^i;\IZ)
\qquad(*\leq2i+1).
\]
Moreover,
\[
H^{2i+3}(\Lambda^i;\IZ)
\cong H^{2i+3}(\Lambda^i,\Lambda^{i-1};\IZ)
\cong\IZ,
\]
$H^*(\Lambda^i;\IZ)=0$ for $*\geq2i+4$, and $H^{2i+2}(\Lambda^i;\IZ)$ is finite cyclic.
\end{corollary}

\begin{proof}
By index monotonicity and the long exact sequence of $(\Lambda^j,\Lambda^i)$, the relative groups $H^*(\Lambda^j,\Lambda^i;\IZ)$ do not contribute in the relevant degrees. Hence
\[
H^*(\Lambda^j;\IZ)\cong H^*(\Lambda^i;\IZ)
\qquad(2j,2i>*,\ i\leq i_0).
\]
Passing to the inverse limit gives
\[
H^*(\Lambda^i;\IZ)\cong H^*(\Lambda M;\IZ)
\qquad(*\leq2i).
\]
The exact sequence of $(\Lambda^i,\Lambda^{i-1})$ further gives
\[
H^{2i+3}(\Lambda^i;\IZ)
\cong H^{2i+3}(\Lambda^i,\Lambda^{i-1};\IZ)
\cong\IZ,
\qquad
H^*(\Lambda^i;\IZ)=0\quad(*>2i+3).
\]
In degree $2i+2$, its exact fragment
\[
H^{2i+2}(\Lambda^i,\Lambda^{i-1};\IZ)
\cong \IZ/\#\pi_1(K_i)\IZ
\longrightarrow H^{2i+2}(\Lambda^i;\IZ)
\longrightarrow H^{2i+2}(\Lambda^{i-1};\IZ)=0
\]
shows that $H^{2i+2}(\Lambda^i;\IZ)$ is finite cyclic. Using this in the long exact sequence of $(\Lambda^{i+1},\Lambda^i)$, namely
\begin{center}
\resizebox{\textwidth}{!}{
\begin{tikzpicture}[commutative diagrams/every diagram]
\matrix[matrix of math nodes,name=m,commutative diagrams/every cell,ampersand replacement=\&] {
H^{2i+1}(\Lambda^{i+1},\Lambda^i;\IZ)=0
\&H^{2i+1}(\Lambda^{i+1};\IZ)\cong\IZ
\&H^{2i+1}(\Lambda^i;\IZ)
\&{}\\
H^{2i+2}(\Lambda^{i+1},\Lambda^i;\IZ)\cong\IZ
\&H^{2i+2}(\Lambda^{i+1};\IZ)\cong\IZ
\&H^{2i+2}(\Lambda^i;\IZ)
\&{}H^{2i+3}(\Lambda^{i+1},\Lambda^i;\IZ)=0\\};
\path[commutative diagrams/.cd,every arrow,every label]
(m-1-1) edge node {$\iota_1^{2i+1}$}(m-1-2)
(m-1-2) edge node {$\iota_2^{2i+1}$}(m-1-3)
(m-1-3) edge (m-1-4)
(m-2-1) edge node {$\iota_1^{2i+2}$}(m-2-2)
(m-2-2) edge node {$\iota_2^{2i+2}$}(m-2-3)
(m-2-3) edge (m-2-4);
\end{tikzpicture}}
\end{center}
we conclude that
\begin{align}\label{firstnontrivialiso}
\iota_2^{2i+1}:H^{2i+1}(\Lambda^{i+1};\IZ)
\xrightarrow{\cong}H^{2i+1}(\Lambda^i;\IZ)
\end{align}
is an isomorphism. Therefore
\[
H^{2i+1}(\Lambda^{i+1};\IZ)\cong H^{2i+1}(\Lambda M;\IZ)
\]
for $i<i_0$.
\end{proof}

\begin{prop}\label{propmappingdegreecorespondence}
For $i\in \IN$ let $B$ be a critical manifold of energy $e_i$. Suppose that $R=\IZ$ or $\IQ$. Let $[M]\in H^3(M,\IZ)$ denote a fundamental class of $M$. Moreover let $p:\Lambda M\rightarrow M$ denote the canonical projection map given by evlautaion at $0$ and $\pi_M:T^1M\rightarrow M$ the canonical projection of the unit tangent bundle. \\

\noindent Then we can choose natural isomorphisms\\ 
$k_1:H^{\ind(B)}(\Lambda^{i},\Lambda^{i-1};R)\rightarrow H^{0}(\Phi(B);R)$ and $k_2:H^{\ind(B)+3}(\Lambda^{i},\Lambda^{i-1};R)\rightarrow H^{3}(\Phi(K);R)$, such that the diagram 
\begin{center}
\begin{tikzpicture}[commutative diagrams/every diagram]
\matrix[matrix of math nodes, name=m, commutative diagrams/every cell] {
H^{\ind(B)}(\Lambda^{i},\Lambda^{i-1};\IZ)& & H^{\ind(B)+3}(\Lambda^{i},\Lambda^{i-1};\IZ)  \\
 H^{0}(\Phi(B);\IZ) & & H^{3}(\Phi(B);\IZ) \\};
\path[commutative diagrams/.cd, every arrow, every label]
(m-1-1) edge node {$\cup p^{*}([M])$}(m-1-3)
edge node {$k_1$}(m-2-1)
(m-1-3) edge node {$k_2$} (m-2-3)
(m-2-1) edge node {$\cdot\cup \pi_M ^{*}([M]) $ } (m-2-3);
\end{tikzpicture}
\end{center}
commutes up to sign.
\end{prop}

\begin{proof}
By Corollary \ref{corconnectedcritsets}, there is at most one critical manifold $B$ of energy $e_i$. The handle-attachment description in Section \ref{sec::topologyfreeloopspace} (see \cite{meyer1967kritische} for more details) provides disc bundles $DN(B)\subset N(B)$ \\
and $D\nu^+(B)\subset \nu^+(B)$ such that $E^{-1}((-\infty,e_i+\varepsilon_i))$ is diffeomorphic to $E^{-1}((-\infty,e_i-\varepsilon_i))$ with the\\ handle $DN(B)\oplus D\nu^{+}(B)$ attached along the boundary of $DN(B)$.\\

\noindent More precisely, there is a homeomorphism $f$ from $DN(B)\oplus D\nu^{+}(B)$ onto a closed subset $h_B\subset E^{-1}(-\infty,e_i+\varepsilon_i)$ satisfying
\begin{enumerate}[label=(\roman*)]
\item $E^{-1}(-\infty,e_i+\varepsilon_i)=E^{-1}(-\infty,e_i-\varepsilon_i)\bigcup h_B$;
\item $f|_{\partial DN(B)\oplus D\nu^{+}(B)}$ is a diffeomorphism onto $h_B\cap\partial E^{-1}(-\infty,e_i-\varepsilon_i)$;
\item $f|_{DN(B)\oplus D\nu^{+}(B)}$ is a diffeomorphism onto
\[
E^{-1}(\infty,e_i+\varepsilon_i)\setminus E^{-1}(-\infty,e_i-\varepsilon_i).
\]
\end{enumerate}
Since $DN(B)\oplus D\nu^{+}(B)$ and $\partial DN(B)\times D\nu^{+}(B)$ are deformation retracts of $DN(B)$ and $\partial DN(B)$, respectively, $E^{-1}(-\infty,e_i-\varepsilon)\bigcup f(DN(B))$ is a deformation retract of
\[
E^{-1}(-\infty,e_i+\varepsilon)=E^{-1}(-\infty,e_i-\varepsilon)\bigcup f(DN(B)\times D\nu^{+}(B)).
\]
Thus the inclusion induces an isomorphism
\begin{align}\label{relativesetisovectorbundle}
H^{*}(E^{-1}(-\infty,e_i+\varepsilon_i),E^{-1}(-\infty,e_i-\varepsilon_i);R)
&\cong H^{*}(DN(B),\partial DN(B);R).
\end{align}
The inclusion
\[
E^{-1}(-\infty,e_i-\varepsilon_i)\cup f(DN(B))\rightarrow E^{-1}(-\infty,e_i+\varepsilon_i)
\]
therefore induces isomorphisms
\begin{align*}
H^{*}(E^{-1}(-\infty,e_i+\varepsilon_i),E^{-1}(-\infty,e_i+\varepsilon_i);\IZ)\rightarrow H^{*}(f(DN(B)),f(\partial DN(B));R)
\end{align*}
for $*=\ind(B),\ind(B)+3$. Let $[M]\in H^{3}(M;\IZ)$ be the fundamental cohomology class of an orientation on $M$. The natural inclusions $\iota_i$ commute with evaluation:
\begin{center}
\begin{tikzpicture}[commutative diagrams/every diagram]
\matrix[matrix of math nodes, name=m, commutative diagrams/every cell] {
\Lambda^{i-1}& E^{-1}([0,e_i-\varepsilon_i))&\Lambda^{i} \\
& M &\\};
\path[commutative diagrams/.cd, every arrow, every label]
(m-1-2) edge node {$\iota_2$} (m-1-3)
edge node {$p$} (m-2-2)
(m-1-1) edge node {$\iota_1$}  (m-1-2)
 edge node[below] {$p$} (m-2-2)
(m-1-3) edge node {$p$}(m-2-2) 
;
\end{tikzpicture}
\end{center}
Naturality of the cup product gives the commutative diagram
\begin{equation}\label{proofdiagrammappingdegreecorres1}
\resizebox{0.92\textwidth}{!}{%
\begin{tikzpicture}[commutative diagrams/every diagram]
\matrix[matrix of math nodes, name=m,column sep=1cm, commutative diagrams/every cell,ampersand replacement =\&] {
H^{*}(\Lambda^{i},\Lambda^{i-1};R)\& H^{*+3}(\Lambda^{i},\Lambda^{i-1};R)\\
H^{*}(E^{-1}(-\infty,e_i+\varepsilon_i),E^{-1}([0,e_i-\varepsilon_i);R) \& H^{*+3}(E^{-1}(-\infty,e_i+\varepsilon_i),E^{-1}([0,e_i-\varepsilon_i);R)\\
H^{*}(f(DN(B)),f(SN(B));R) \& H^{*+3}(f(DN(B)),f(SN(B));R)\\
H^{*}(DN(B),SN(B);R) \& H^{*+3}(DN(B),SN(B),R)\\};
\path[commutative diagrams/.cd, every arrow, every label]
(m-1-1) edge node {$ \cup p^{*}([M])$} (m-1-2)
(m-2-1) edge node {$ \cup p^{*}([M])$}(m-2-2)
(m-3-1) edge node {$ \cup p^{*}([M])$}(m-3-2)
(m-4-1) edge node {$ \cup f^{*}(p^{*}([M]))$}(m-4-2)

(m-2-1) edge[right] node {$\iota^{*} _3$} (m-1-1)
(m-2-2) edge[right] node {$\iota^{*} _3$} (m-1-2)
(m-2-1) edge node {$\iota^{*} _4$} (m-3-1)
(m-2-2) edge node {$\iota^{*} _4$} (m-3-2)
(m-3-1) edge node {$f^{*}$}(m-4-1)
(m-3-2) edge node {$f^{*}$}(m-4-2);
\end{tikzpicture}%
}
\end{equation}
in which all vertical maps are isomorphisms. Let
\[
th(N(B))\in H^{\ind(B)}(DN(B),SN(B);R)
\]
be the Thom class of the negative bundle. The cup-product identities show that the Thom isomorphism commutes, up to sign, with cup product by $f^{*}(p^{*}([M]))$. Since $f|_{B}\equiv\id_{B}$, this compatibility extends along the bundle projection $N(B)\rightarrow B$, with cup product by $p^{*}([M])$ on $B$. Together with the commutative triangle
\begin{center}
\begin{tikzpicture}[commutative diagrams/every diagram]
\matrix[matrix of math nodes, name=m, commutative diagrams/every cell] {
 B & & C \\
& M & \\};
\path[commutative diagrams/.cd, every arrow, every label]
(m-1-1) edge node {$\Phi $} (m-1-3)
(m-1-1) edge[below] node {$p$}  (m-2-2) 
(m-1-3) edge node {$\pi_M$}(m-2-2);
\end{tikzpicture}
\end{center}
this gives the diagram
\begin{equation}\label{proofdiagrammappingdegreecorres2}
\begin{tikzpicture}[commutative diagrams/every diagram]
\matrix[matrix of math nodes, name=m,column sep= 2cm, commutative diagrams/every cell] {
H^{\ind(B)}(DN(B),SN(B);R) & H^{\ind(B)+3}(DN(B),SN(B);R)\\
H^{0}(DN(B);R) & H^{3}(DN(B);R)\\
H^{0}(B;R)& H^{3} (B;R)\\
H^{0}(C;R)& H^{3} (C;R)\\};

\path[commutative diagrams/.cd, every arrow, every label]
(m-1-1) edge node {$ \cup f^{*}(p^{*}([M]))$}(m-1-2)
(m-2-1) edge node {$ \cup f^{*}(p^{*}([M]))$}(m-2-2)
(m-3-1) edge node {$ \cup p^{*}([M]) $} (m-3-2)
(m-4-1) edge node {$ \cup \pi_M ^{*}([M])$}(m-4-2)

(m-2-1) edge[right] node {$\cup th(N(B))$}(m-1-1)
(m-2-2) edge[right] node {$\cup th(N(B))$}(m-1-2)
(m-3-1) edge[right] node {$(N(B)\rightarrow B)^{*}$}(m-2-1)
(m-3-2) edge[right] node {$(N(B)\rightarrow B)^{*}$}(m-2-2)
(m-4-1) edge[right] node {$\Phi^{*} $} (m-3-1)
(m-4-2) edge[right] node {$\Phi^{*} $} (m-3-2);
\end{tikzpicture}
\end{equation}
which commutes up to sign and has isomorphisms in both columns. Define $k_1$ and $k_2$ by composing the vertical maps in the first and second columns of (\ref{proofdiagrammappingdegreecorres1}) and (\ref{proofdiagrammappingdegreecorres2}), using inverses where necessary. Commutativity up to sign proves the claim.
\end{proof}

\begin{corollary}\label{nonzerodegree}
Let $ K_i\subset \Lambda M $ be a critical manifold of the Energy functional. Then the induced map
\begin{align*}
\pi_M |_{\Phi(K_i)} ^{*}:H^{3}(M;\IZ)\rightarrow H^{3}(\Phi(K_i);\IZ)
\end{align*}
is non-zero.
\end{corollary}

\begin{proof}
If $\Phi(K_i)=T^1M$, the claim follows from $T^1M\cong \IS^3\times\IS^2$. Suppose therefore that $K_i$ is singular. By rational perfectness, in degrees $m$ and $m+3$ the inclusion $\Lambda^i\hookrightarrow\Lambda M$ and the canonical maps
\begin{align*}
H^{*}(\Lambda^i,\Lambda^{i-1};\IQ)\longrightarrow H^{*}(\Lambda^i;\IQ)
\end{align*}
are isomorphisms. Since $M\cong\IS^3$,
\begin{align*}
H^{*}(\Lambda M;\IQ)\cong H^{*}(\Omega M;\IQ)\otimes_{\IQ}H^{*}(M;\IQ)
\end{align*}
as algebras, and multiplication by $p^*([M])$ is an isomorphism from even degree $*$ to degree $*+3$. In particular, since $m$ is even, naturality gives an isomorphism
\begin{align*}
\cup p^{*}([M]):
H^{\ind(K_i)}(\Lambda^i,\Lambda^{i-1};\IQ)
\xrightarrow{\cong}
H^{\ind(K_i)+3}(\Lambda^i,\Lambda^{i-1};\IQ).
\end{align*}
By Proposition \ref{propmappingdegreecorespondence}, this map corresponds, up to sign, to multiplication by $\pi_M^*([M])$ on $H^0(\Phi(K_i);\IQ)$. Hence $\pi_M^*([M])\neq0$ in $H^3(\Phi(K_i);\IQ)$, and therefore also in $H^3(\Phi(K_i);\IZ)$.
\end{proof}

\section{The \texorpdfstring{$\IS^1$}{S1}-equivariant cohomology of iterates of critical manifolds}

We compute the $\IS^1$-equivariant cohomology of iterated critical manifolds. We first record a general consequence of the Gysin sequence and then specialize to critical manifolds consisting of prime geodesics.
\begin{prop}\label{propprilimcritcohomcomp}
If $K$ is a critical manifold, then $H^{2*+1} _{\IS^1}(K;\IZ)=0$. If $K$ is regular $H^{2} _{\IS^1} (K;\IZ)\cong \IZ^2$ and $$\cup e:H^{2*} _{\IS^1}(K;\IZ)\rightarrow H^{2*+2} _{\IS^1}(K;\IZ)$$ is an isomorphism for $*\geq 5$. Moreover if $K$ is singular, then $$\cup e:H^{2*} _{\IS^1}(K;\IZ)\rightarrow H^{2*+2} _{\IS^1}(K;\IZ)$$ is an isomorphism for $*\geq 3$. 
\end{prop}

\begin{proof}
Consider first the case, that $K$ is regular, in which $K\cong \IS^3\times \IS^2$. A first application of the Gysin sequence gives us directly that  
\begin{center}
\begin{tikzpicture}[commutative diagrams/every diagram]
\matrix[matrix of math nodes, name=m,column sep = 0.35cm, commutative diagrams/every cell] {
H^{-1} _{\IS^1}(K;\IZ)\cong 0& H^{1} _{\IS^1}(K;\IZ)& H^{1} (K;\IZ)\cong 0,\\
};
\path[commutative diagrams/.cd, every arrow, every label]
(m-1-1) edge node {$\cup e $} (m-1-2)
(m-1-2) edge (m-1-3);
\end{tikzpicture}
\end{center}
so $H^1 _{\IS^1}(K;\IZ)\cong 0$ and we obtain in this case also 
\begin{center}
\begin{tikzpicture}[commutative diagrams/every diagram]
\matrix[matrix of math nodes, name=m,column sep = 0.35cm, commutative diagrams/every cell] {
H^{1} (K;\IZ)\cong 0& H^{0} _{\IS^1}(K;\IZ)\cong \IZ& H^{2}  _{\IS^1}(K;\IZ)&H^2(K;\IZ)\cong \IZ&H^1 _{\IS^1}(K;\IZ)\cong 0,\\
};
\path[commutative diagrams/.cd, every arrow, every label]
(m-1-1) edge  (m-1-2)
(m-1-2) edge node {$\cup e $}(m-1-3)
(m-1-3) edge (m-1-4)
(m-1-4) edge (m-1-5);
\end{tikzpicture}
\end{center}
So this short exact sequence splits and we obtain an isomorphism $H^{2} _{\IS^1}(K;\IZ)\cong \IZ$. Together with 
\begin{center}
\begin{tikzpicture}[commutative diagrams/every diagram]
\matrix[matrix of math nodes, name=m,column sep = 0.35cm, commutative diagrams/every cell] {
H^{1} _{\IS^1}(K;\IZ)\cong 0& H^{3} _{\IS^1}(K;\IZ)& H^{3}  (K;\IZ)\cong \IZ&H^2 _{\IS^1}(K;\IZ)\cong \IZ^2&H^4 _{\IS^1}(K;\IZ)& H^4(K;\IZ)\cong0\\
};
\path[commutative diagrams/.cd, every arrow, every label]
(m-1-1) edge  (m-1-2)
(m-1-2) edge node {$\cup e $}(m-1-3)
(m-1-3) edge (m-1-4)
(m-1-4) edge (m-1-5)
(m-1-5) edge (m-1-6);
\end{tikzpicture}
\end{center}
we see that $H^{3} _{\IS^1}(K;\IZ)$ is free and $H^{4}(K;\IZ)$ is finitely generated. By exactness we see that $H^{5} _{\IS^1}(K;\IZ)$ is finitely generated as well. Analogously by considering the Gysin sequence for $\IQ$-coefficients and using the same arguments (replace every $\IZ$ in the proof so far by a $\IQ$) we obtain the same results for $\IQ$-coefficients. Since $\IS^1$ acts orientation preserving and effectively on $T^1M$, we see that $K/\IS^1$ becomes an orientable $4$-orbifold (In fact we will later see that it even possesses the structure of a smooth manifold) and the canonical projection $q:K\times_{\IS^1}E\IS^1\rightarrow K/\IS^1$ becomes a rational homotopy equivalence. Hence $H^{*} _{\IS^1}(K;\IQ)\cong H^{*}(K/\IS^1;\IZ).$ By the universal coefficient theorem we furthermore have 
$$H^*(K/\IS^1;\IR)\cong \Hom (H_{*}(K/\IS^1;\IQ);\IR)\cong \Hom (H^{*} (K/\IS^1;\IQ);\IR).$$ By the deRham isomorphsim for orbifolds and the Poincare duality for orbifolds, that $H^{5} (K/\IS^1;\IR)\cong 0,H^{3}(K/\IS^1;\IR)\cong H^{1}(K/\IS^1;\IR)\cong 0$ and $H^{4}(K/\IS^1;\IR)\cong H^{0}(K/\IS^1;\IR)\cong \IR$. From this we deduce with $H^{*}(K/\IS^1;\IR)\cong \Hom(H^{*}(K/\IS^1;\IQ);\IR)$ and the fact that $H^{*} _{\IS^1}(K;\IQ)$ is finitely generated, so  of finite dimension, that $H^{5}(K/\IS^1)\cong 0, H^{3}(K/\IS^1;\IQ)\cong 0$ and $H^4(K/\IS^1;\IQ)\cong \IQ$. It follows that $H^{5} _{\IS^1}(K;\IQ)\cong H^{3} _{\IS^1}(K;\IQ)\cong 0 $ and $H^{4} _{\IS^1}(K;\IQ)\cong \IQ$. By applying the universal coefficient theorem for cohomology once again we get, since $H^{*}(K;\IZ)$ is free and finitely generated for $*\leq 3$, that $H^{3}(K;\IZ)\cong 0$ and the free part of $H^{4} _{\IS^1}(K;\IZ)$ is isomorphic to $\IZ$. By the exactness of the sequence 
\begin{equation}\label{prooffirstcohomologycomputiationiterates}
\begin{tikzpicture}[commutative diagrams/every diagram]
\matrix[matrix of math nodes, name=m,column sep = 0.35cm, commutative diagrams/every cell] {
 H^{3} _{\IS^1}(K;\IZ) \cong 0& H^{5} _{\IS^1}  (K;\IZ)&H^5(K;\IZ)\cong \IZ\\
};
\path[commutative diagrams/.cd, every arrow, every label]
(m-1-1) edge node {$\cup e $} (m-1-2)
(m-1-2) edge (m-1-3);
\end{tikzpicture}
\end{equation} 
we see that $H^{5} _{\IS^1}(K;\IZ)$ is free. Since $H^{5}(K;\IQ)$ is zero, this implies with the universal coefficient theorem, that $H^{5} _{\IS^1}(K;\IZ)$ is zero. By the long exact sequence for $*\geq 3$ 
\begin{center}
\begin{tikzpicture}[commutative diagrams/every diagram]
\matrix[matrix of math nodes, name=m,column sep = 0.35cm, commutative diagrams/every cell] {
 H^{2*w} (K;\IZ) \cong 0& H^{2*-1} _{\IS^1}  (K;\IZ)&H^{2*+1} _{\IS^1}(K;\IZ) & H^{2*+1} (K;\IZ)\\
};
\path[commutative diagrams/.cd, every arrow, every label]
(m-1-1) edge  (m-1-2)
(m-1-2) edge node {$\cup e $}(m-1-3)
(m-1-3) edge (m-1-4);
\end{tikzpicture}
\end{center} 
it follows inductively that $H^{2*+1} _{\IS^1}(K;\IZ)\cong 0$ for all $*$. If $K$ is singular, then $K$ has the integral cohomology ring of a lens space. As in the first case we obtain that  
\begin{center}
\begin{tikzpicture}[commutative diagrams/every diagram]
\matrix[matrix of math nodes, name=m,column sep = 0.35cm, commutative diagrams/every cell] {
 H^{-1} _{\IS^1}(K;\IZ) \cong 0& H^{1} _{\IS^1}  (K;\IZ)&H^1(K;\IZ)\cong 0\\
};
\path[commutative diagrams/.cd, every arrow, every label]
(m-1-1) edge node {$\cup e $} (m-1-2)
(m-1-2) edge (m-1-3);
\end{tikzpicture}
\end{center}
so $H^{1} _{\IS^1}(K;\IZ)\cong 0$. Furthermore we see, from the exact sequences
\begin{center}
\begin{tikzpicture}[commutative diagrams/every diagram]
\matrix[matrix of math nodes, name=m,column sep = 0.35cm, commutative diagrams/every cell] {
H^1(K;\IZ)\cong 0  & H^{0} _{\IS^1}(K;\IZ) \cong 0& H^{2} _{\IS^1}  (K;\IZ)& H^2(K;\IZ)\cong \IZ/l\IZ & H^1(K;\IZ)\cong 0,\\
};
\path[commutative diagrams/.cd, every arrow, every label]
(m-1-1) edge  (m-1-2)
(m-1-2) edge node {$\cup e $}(m-1-3)
(m-1-3) edge (m-1-4)
(m-1-4) edge (m-1-5);
\end{tikzpicture}
\end{center}
\begin{center}
\begin{tikzpicture}[commutative diagrams/every diagram]
\matrix[matrix of math nodes, name=m,column sep = 0.35cm, commutative diagrams/every cell] {
 H^{1} _{\IS^1} (K;\IZ) \cong 0& H^{3} _{\IS^1}  (K;\IZ)& H^{3} _{\IS^1}(K;\IZ)\cong \IZ&H^{2}(K;\IZ)\cong \IZ/l\IZ\\
};
\path[commutative diagrams/.cd, every arrow, every label]
(m-1-1) edge node {$\cup e $} (m-1-2)
(m-1-2) edge (m-1-3)
(m-1-3) edge node {$0$} (m-1-4);
\end{tikzpicture}
\end{center}
and 
\begin{center}
\begin{tikzpicture}[commutative diagrams/every diagram]
\matrix[matrix of math nodes, name=m,column sep = 0.35cm, commutative diagrams/every cell] {
 H^{*+1} (K;\IZ) \cong 0& H^{*} _{\IS^1}  (K;\IZ)& H^{*+2} _{\IS^1}(K;\IZ)&H^{*+1}(K;\IZ)\cong 0\\
};
\path[commutative diagrams/.cd, every arrow, every label]
(m-1-1) edge  (m-1-2)
(m-1-2) edge node {$\cup e $}(m-1-3)
(m-1-3) edge (m-1-4);
\end{tikzpicture}
\end{center}
for $*\geq 3$. We see that, that $H^{3} _{\IS^1}(K;\IZ)$ is free and finitely generated, $H^2 _{\IS^1}(K;\IZ)$ is finitely generated and $H^{*} _{\IS^1}(K;\IZ)$ is finitely generated for $*\geq 3$. Since $K$ is a rational cohomology sphere we obtain by considering the same sequences as above in rational cohomology, that $$H^{*} _{\IS^1}(K;\IQ)=\begin{cases} \IQ,~~*=0,2\\
0,~~\textup{else}.
\end{cases}$$
Since $H^{3} _{\IS^1}(K;\IZ)$ is free and finitely generated, this implies by the universal coefficient theorem for cohomology, that $H^{3} _{\IS^1}(K;\IZ)\cong 0$. Together with the sequence (\ref{prooffirstcohomologycomputiationiterates}) we obtain as in the previous case $H^{2*+1} _{\IS^1}(K;\IZ)\cong 0$ for all $*\in \IZ$.
\end{proof}

Now let $K$ consist only of prime geodesics. If this is the case, the $\IS^1$-action on $K$ is free, so
\[
K\longrightarrow K/\IS^1
\]
is a principal $\IS^1$-bundle with Euler class $e(K)$.\\

\noindent There are now two cases we will consider, namely that $K$ is singular or that $K$ is already regular and we are given the round metric on the $3$-sphere. If $K$ is singular, then $K$ has the integral cohomology ring of a lens space and $K/\IS^1\cong \IS^2$; rational perfectness excludes positive genus. In this case
\[
\langle e(K),[K/\IS^1]\rangle=\#\pi_1(K).
\]
The other case we are considering is the case of the round metric on the three sphere. In this case we obtain diffeomorphisms $K\cong T^1\IS^3$ and $K/\IS^1\cong\IS^2\times\IS^2$ (see \cite{besse2012manifolds}). If $x_1,x_2$ denote the two canonical generators of $H^2(\IS^2\times\IS^2;\IZ)$, the Weinstein number (see \cite[Theorem~2.21]{besse2012manifolds} or \cite[Theorem~B]{weinstein1974volume})
\[
\langle e(K)^2,[K/\IS^1]\rangle=2
\]
implies, after choosing orientations of the two factors, that
\[
e(K)=x_1+x_2.
\]\\

\noindent For $q\in\IN$, the Borel construction of the iterate can be written as
\[
K^q\times_{\IS^1}E\IS^1\cong K\times_{\IS^1}B(\IZ/q\IZ).
\]
The residual free circle action gives a principal bundle
\[
\pi:K\times_{\IS^1}B(\IZ/q\IZ)\longrightarrow K/\IS^1\times B\IS^1.
\]
Thus $H^*_{\IS^1}(K^q;\IZ)$ is determined by the Gysin sequence of $\pi$ once its Euler class is known.

\begin{prop}\label{propcomputationeulerclass}
Let $e(K)\in H^2(K/\IS^1;\IZ)$ be the Euler class of the bundle $K\rightarrow K/\IS^1$ and let $c\in H^2(B\IS^1;\IZ)$ be the first chern class.
Then the Euler class $e(\pi)\in H^2(K/\IS^1\times B\IS^1;\IZ)$ is given by
$$e(K)\otimes 1 -q(1\otimes c). $$
\end{prop}

\begin{proof}
Up to isomorphy there is a unique map of principal $\IS^1$-bundles  
$F:K \times_{\IS^1} B(\IZ/q\IZ)\rightarrow E\IS^1$ such that the following diagram commutes
\[
\begin{tikzcd}[column sep=large, row sep=large]
K\times_{\IS^1} B(\IZ/q\IZ)
    \arrow{r}{F}
    \arrow{d}{\pi}
&
E\IS^1
    \arrow{d}{\pi_{\mathrm{univ}}}
\\
K/\IS^1 \times B\IS^1
    \arrow{r}{f}
&
B\IS^1
\end{tikzcd}
\]
The Euler class of the bundle $\pi$ is then given by $e(\pi)=f^{*}(c)$. For each $p\in K$ and each $q\in B\IZ/q\IZ$ we obtain by the assignment $x\mapsto [p,x]$ an $\IS^1$-equivariant map $\iota_{K}:K\rightarrow K\times_{\IS^1}B(\IZ/q\IZ)$ and by the assignment $x\mapsto [x,q]$ an a $\IS^1$-anti equivariant map  $\iota_{B\IZ/q\IZ}:B(\IZ/q\IZ )\rightarrow K\times_{\IS^1}B(\IZ/q\IZ)$ who induce maps $\iota_{K/\IS^1}:K/\IS^1 \rightarrow K/\IS^1\times B\IS^1$ and $\iota_{B\IS^1}:B\IS^1 \rightarrow K/\IS^1\times B\IS^1$.

 such that the following diagram commute
 \[
\begin{tikzcd}[column sep=large, row sep=large]
K
    \arrow{r}{F_1}
    \arrow{d}{\pi\circ \iota_{K}}
&
E\IS^1
    \arrow{d}{\pi_{\mathrm{univ}}}
&&
B(\IZ/q\IZ)
    \arrow{r}{F_2}
    \arrow{d}{\pi\circ \iota_{B(\IZ/q\IZ)}}
&
E\IS^1
    \arrow{d}{\pi_{\mathrm{univ}}}
\\
K/\IS^1
    \arrow{r}{f\circ\iota_{K/\IS^1}}
&
B\IS^1
&&
B\IS^1
    \arrow{r}{f\circ\iota_{B\IS^1}}
&
B\IS^1
\end{tikzcd}
\] 
This gives us a $\IS^1$-equivariant map $F_1=F\circ \iota _{K}:K\times E\IS^1\rightarrow E\IS^1$ and a $\IS^1$-anti \\equivariant map $F_2=F\circ \iota_{B\IZ/q\IZ}:B(\IZ/q\IZ)\rightarrow E\IS^1$. The map $F_1$ is the bundle map from the bundle $K\rightarrow K/\IS^1$ and $(f\circ \iota_{K/\IS^1})^{*}(c)$ is the euler class of the bundle $K\rightarrow K/\IS^1$.
The map $F_2$ is a bundle map from the sphere bundle of the line bundle $\bar{\IC} \times_{\IS^1} B(\IZ/q\IZ)$, which is the conjugate of the line bundle $\IC \times_{\IS^1} B(\IZ/q\IZ)$. Since the euler class of the bundle $\IC \times_{\IS^1}E\IS^1$ is $q$-times the first chern class $c\in H^2(B\IS^1;\IZ)$, the bundle $\bar{\IC}\times_{\IS^1}B\IZ/q\IZ$ has euler class given by $-qc$. By the Universal coefficient theorem, the second cohomology of the base is given by
\begin{align*}
H^2(K/\IS^1\times B\IS^1;\IZ)\cong H^2(K/\IS^1;\IZ)\otimes H^0(B\IS^1;\IZ)\oplus H^0(K/\IS^1;\IZ)\otimes H^2(B\IS^1;\IZ).
\end{align*}
The Euler class $e(\pi)=f^{*}(c)$ is then given by 
\begin{align*}
e(\pi)&=f^{*}(c)\\
&=\iota^{*} _{K/\IS^1}f^{*}(c)\otimes 1+1\otimes \iota^{*} _{B\IS^1} f^{*}(c)\\
&=e(\pi\circ \iota_{K/\IS^1})\otimes 1+1\otimes (e(\pi\circ \iota_{B\IS^1}))\\
&=e(K)\otimes 1 -q( 1\otimes c).
\end{align*}
\end{proof}

For the round sphere,
\[
H^*(K/\IS^1\times B\IS^1;\IZ)
\cong \IZ[x_1,x_2,c]/(x_1^2,x_2^2),
\qquad |x_1|=|x_2|=|c|=2,
\]
and Proposition \ref{propcomputationeulerclass} gives $e(\pi)=x_1+x_2-qc$. Hence multiplication by $e(\pi)$ is represented, for $k\geq3$, by
\[
\begin{pmatrix}
-q&1&1&0\\
0&-q&0&1\\
0&0&-q&1\\
0&0&0&-q
\end{pmatrix}
\]
with respect to the ordered bases
\[
\{x_1x_2c^{k-3},x_1c^{k-2},x_2c^{k-2},c^{k-1}\},\qquad
\{x_1x_2c^{k-2},x_1c^{k-1},x_2c^{k-1},c^k\}.
\]
For $k=2$, with domain basis $\{x_1,x_2,c\}$ and codomain basis $\{x_1x_2,x_1c,x_2c,c^2\}$, the matrix is
\[
\begin{pmatrix}
1&1&0\\
-q&0&1\\
0&-q&1\\
0&0&-q
\end{pmatrix}.
\]
The Gysin sequence and the cokernels of these matrices yield the following result.

\begin{theorem}\label{thmcohomologycritmanifolds}
Let $(M,g)$ is isometric to the round sphere. Let $K\cong T^1 M $ be the critical manifold of lowest positive energy. Let $q\in \IN $. Then the $\IS^1$-equivariant cohomology with $\IZ$ coefficients of the critical manifold $K^q$ is given by 
\begin{align*}
H^{*} _{\IS^1}(K^q;\IZ)\cong \begin{cases}
\IZ,~~&*=0,\\
\IZ^2,~~&*=2,\\
\IZ^2\oplus \IZ/2q\IZ,~~&*=4,\\
\IZ/2q\IZ \oplus \IZ/\frac{q}{2}q^2\IZ,~~&*>4~~~~*\equiv 0\mod 2,\\
0~~&\textup{else},
\end{cases}
\end{align*}
if $q$ is even and
\begin{align*}
H^{*} _{\IS^1}(K^q;\IZ)\cong \begin{cases}
\IZ,~~&*=0,\\
\IZ^2,~~&*=2,\\
\IZ^2\oplus \IZ/q\IZ,~~&*=4,\\
\IZ/q\IZ\oplus \IZ/q^3\IZ,~~&*>4,~~~~ *\equiv 0\mod 2,\\
0~~&\textup{else},
\end{cases}
\end{align*}  
if $q$ is odd.
\end{theorem}

If $K$ is singular, write $\alpha=\#\pi_1(K)$. Since $K/\IS^1\cong\IS^2$,
\[
H^*(K/\IS^1\times B\IS^1;\IZ)
\cong \IZ[x,c]/(x^2),\qquad |x|=|c|=2,
\]
and $e(\pi)=\alpha x-qc$. Multiplication by $e(\pi)$ is represented in degree $0\to2$ by
\[
\begin{pmatrix}\alpha\\-q\end{pmatrix},
\]
and in the subsequent even degrees, with the natural bases $\{xc^{k-2},c^{k-1}\}$ and $\{xc^{k-1},c^k\}$, by
\[
\begin{pmatrix}
-q&\alpha\\
0&-q
\end{pmatrix}.
\]
The Gysin sequence therefore gives the following theorem.

\begin{theorem}
Let $(M,g)$ be a simply connected Riemannian $3$-manifold, all of whose geodesics are closed. Suppose, that $K\subset \Lambda M$ is a singular critical manifold containing only prime geodesics. Then for any $q\in \IN$ the $\IS^1$-equivariant cohomology with $\IZ$ coefficients of the critical manifold $K^q$ is given by
\begin{align*}
H^{*} _{\IS^1}(K^q;\IZ)\cong \begin{cases}
\IZ,~~&*=0,\\
\IZ\oplus \IZ/(q,\#\pi_1(K))\IZ,~~&*=2,\\
\IZ/(q,\#\pi_1(K))\IZ \oplus \IZ/\frac{q^2}{(q,\#\pi_1(K))}\IZ,~~&*>2,~~*\equiv 0\mod 2\\
0,~~&\textup{else}.
\end{cases}
\end{align*}
\end{theorem}

\section{Order of the isotropy groups of the geodesic flow}

We show that all isotropy groups of the geodesic flow have odd order. Let $i_0$ be minimal such that $K_{i_0}$ is regular. Since there is at most one critical manifold at each critical energy, the number of positive-energy critical manifolds below $K_{i_0}$ is
\begin{align}\label{connectionnumberisotropygroupsparity}
i_0-1=\#\{K_i^l:l\leq q_i,\ 1\leq i\leq k\}=\sum_{i=1}^k(q_i-1).
\end{align}
Thus, if $i_0-1$ is even, all isotropy groups have odd order. We first show that at most one isotropy group can have even order and that, if $i_0$ is even, then $H^3_{\IS^1}(K_{i_0};\IZ/2\IZ)\neq0$. We then use the known $\IS^1$-equivariant cohomology of $\Lambda M$ with $\IZ/2\IZ$ coefficients to derive a recursion for its Betti numbers. Together with the vanishing of $H^{2i_0}(\Lambda^{i_0-1};\IZ)$ and Morse theory, this gives the opposite conclusion and hence a contradiction.

\begin{prop}\label{propatmostoneevenisoorder}
Let $q_1,\ldots,q_k$ be the isotropy orders of $C_1,\ldots,C_k$. At most one $q_i$ is even.
\end{prop}

\begin{proof}
Assume that there exist $i,j\in\{1,\ldots,k\}$ with $i\neq j_1$ and $q_j\equiv q_{j_1}\equiv0\mod2$. Then $\tilde K_i^{q_i/2}$ and $\tilde{K}_j^{q_j/2}$ are singular and have the same energy. By index monotonicity,
\[
\ind(\tilde{K}_i^{q_i})=\ind(\tilde{K}_j^{q_j})<\ind(K_{i_0}),
\]
contradicting the uniqueness of a connected singular critical manifold of a given index.
\end{proof}

\begin{lemma}\label{lemregularleveltwotorsion}
Let $i_0$ be minimal such that $K_{i_0}$ is regular. If $i_0\equiv0\mod2$, then
\[
H^3_{\IS^1}(K_{i_0};\IZ/2\IZ)\neq0.
\]
\end{lemma}

\begin{proof}
By Proposition \ref{propprilimcritcohomcomp}, $H^{2*+1}_{\IS^1}(K_{i_0};\IZ)=0$ and $H^*_{\IS^1}(K_{i_0};\IZ)$ is finitely generated. Hence the universal coefficient theorem (in the form of Theorem 10 in Chapter 5 of \cite{spanier1989algebraic}) gives
\[
H^3(K_{i_0};\IZ/2\IZ)\cong
Tor_1(H^4_{\IS^1}(K_{i_0};\IZ);\IZ/2\IZ),
\]
so it suffices to show that $H^4_{\IS^1}(K_{i_0};\IZ)$ has a cyclic factor of even order.

Choose $\IS^1$-invariant neighborhoods $U_1,\ldots,U_k$ of $\tilde K_1^{q_1},\ldots,\tilde K_k^{q_k}\subset K_{i_0}$. Since $\IS^1$ acts freely on $K_{i_0}\setminus\bigcup_i\tilde K_i^{q_i}$, both relevant quotients have dimension at most $4$, and therefore, for $*\geq5$,
\[
H^*_{\IS^1}\!\left(K_{i_0}\setminus\bigcup_{i=1}^k\tilde K_i^{q_i};\IZ\right)
\cong
H^*\!\left(\left(\bigcup_{i=1}^kU_i\setminus\bigcup_{i=1}^k\tilde K_i^{q_i}\right)/\IS^1;\IZ\right)=0.
\]
Mayer--Vietoris thus yields
\[
H^6_{\IS^1}(K_{i_0};\IZ)
\cong H^6_{\IS^1}\!\left(\bigcup_{i=1}^kU_i;\IZ\right)
\cong\bigoplus_{i=0}^k H^6_{\IS^1}(\tilde K_i^{q_i};\IZ).
\]

If $i_0\equiv0\mod2$, then $i_0-1$ is odd, and by \eqref{connectionnumberisotropygroupsparity},
\begin{align*}
\sum_{i=1}^k(q_i-1)
&=\#\{K\subset\Lambda M:K\textup{ is a critical manifold of }E,\ \ind(K)<\ind(K_{i_0})\}\\
&=i_0-1\equiv1\mod2.
\end{align*}
Hence $q_j-1$ is odd for some $j$, and, since the fundamental group of each singular critical manifold has even order,
\[
1<2\mid(q_j,\#\pi_1(\tilde K_j^{q_j})),
\qquad
2\mid q_j\mid\frac{q_j^2}{(q_j,\#\pi_1(\tilde K_j^{q_j}))}.
\]
Theorem \ref{thmcohomologycritmanifolds} gives
\[
H^6_{\IS^1}(\tilde K_j^{q_j};\IZ)
\cong
\IZ/(q_j,\#\pi_1(\tilde K^{q_j}))\IZ
\oplus
\IZ/\frac{q_j^2}{(q_j,\#\pi_1(\tilde K_j^{q_j}))}.
\]
Since $H^7_{\IS^1}(\tilde K_j^{q_j};\IZ)=0$, the coefficient theorem implies
\[
H^6_{\IS^1}(\tilde K_j^{q_j};\IZ/2\IZ)
\cong H^6_{\IS^1}(\tilde K_j^{q_j};\IZ)\otimes\IZ/2\IZ
\cong(\IZ/2\IZ)^2.
\]
The Gysin sequence therefore gives a surjection
\[
H^4_{\IS^1}(K_{i_0};\IZ/2\IZ)
\longrightarrow H^6_{\IS^1}(\tilde K_j^{q_j};\IZ/2\IZ).
\]
If $H^4_{\IS^1}(K_{i_0};\IZ)$ had no cyclic factor of even order, then, using the universal coefficient theorem and that the free part $Free(H^4_{\IS^1}(K_{i_0};\IZ))$ of $H^4_{\IS^1}(K_{i_0};\IZ)$ is isomorphic to $\IZ$,
\begin{align*}
H^4_{\IS^1}(K_{i_0};\IZ/2\IZ)
&\cong H^4_{\IS^1}(K_{i_0};\IZ)\otimes\IZ/2\IZ
   \oplus Tor_1(H^5_{\IS^1}(K_{i_0};\IZ);\IZ/2\IZ)\\
&\cong H^4_{\IS^1}(K_{i_0};\IZ)\otimes\IZ/2\IZ
\cong Free(H^4_{\IS^1}(K_{i_0};\IZ))\otimes\IZ/2\IZ
\cong\IZ/2\IZ,
\end{align*}
contradicting the surjectivity of
\[
\cup e:H^4_{\IS^1}(K_{i_0};\IZ/2\IZ)
\longrightarrow H^6_{\IS^1}(K_{i_0};\IZ/2\IZ).
\]
\end{proof}

\begin{lemma}\label{lastevensingularlevelzero}
If $i_0$ is minimal such that $K_{i_0}$ is regular, then
\[
H^{2i_0}(\Lambda^{i_0-1};\IZ)\cong0.
\]
\end{lemma}

\begin{proof}
It suffices to show that the inclusion induces an isomorphism
\[
H^{2i_0}(\Lambda^{i_0},\Lambda^{i_0-1};\IZ)
\longrightarrow H^{2i_0}(\Lambda^{i_0};\IZ),
\]
because then the result follows from the long exact sequence of the pair and $H^{2i_0+1}(\Lambda^{i_0},\Lambda^{i_0-1};\IZ)=0$. Consider
\begin{equation}\label{diagramlastsublevelsetproof}
\resizebox{\textwidth}{!}{
\begin{tikzpicture}[commutative diagrams/every diagram]
\matrix[matrix of math nodes, name=m, column sep=.5cm, commutative diagrams/every cell, ampersand replacement=\&] {
{} \& H^{2i_0+3}(\Lambda^{i_0},\Lambda^{i_0-1};\IZ) \& H^{2i_0+3}(\Lambda^{i_0};\IZ) \& {} \\
{} \& H^{2i_0}(\Lambda^{i_0},\Lambda^{i_0-1};\IZ) \& H^{2i_0}(\Lambda^{i_0};\IZ) \& H^{2i_0}(\Lambda^{i_0-1};\IZ) \& H^{2i_0+1}(\Lambda^{i_0},\Lambda^{i_0-1};\IZ)\cong0\\
};
\path[commutative diagrams/.cd, every arrow, every label]
(m-1-1) edge node{$0$} (m-1-2)
(m-1-2) edge (m-1-3)
(m-1-3) edge node{$0$} (m-1-4)
(m-2-1) edge (m-2-2)
(m-2-2) edge (m-2-3)
(m-2-3) edge (m-2-4)
(m-2-4) edge (m-2-5)
(m-2-2) edge node{$\cup p^*([M])$} (m-1-2)
(m-2-3) edge node{$\cup p^*([M])$} (m-1-3);
\end{tikzpicture}}
\end{equation}
The rows come from the long exact sequence of the pair and the vertical maps are cup product with $p^*([M])$. We show that both vertical maps are isomorphisms.

For the second column, the diffeomorphism $M\cong\IS^3$ and the decomposition
\[
\Lambda=\Omega_mM\times M
\]
imply by K\"unneth that
\begin{align}\label{equationfundcupfreeloopspace}
\cdot\cup p^*([M]):H^{2*}(\Lambda M;\IZ)\longrightarrow H^{2*+3}(\Lambda M;\IZ)
\end{align}
is an isomorphism for every $*\in\IZ$. Since $K_{i_0}$ is regular, the index gap theorem in \cite{radeschi2017berger} gives no critical manifold of index $\ind(K_{i_0})+2$ and hence $\ind(K_{i_0+1})=2i_0+4$. Therefore the inclusions induce isomorphisms
\begin{align}
H^*(\Lambda M;\IZ)&\longrightarrow H^*(\Lambda^{i_0};\IZ), && *\leq2i_0+2,\label{equationisomorphimfinalloopspacegroup}\\
H^*(\Lambda M;\IZ)&\longrightarrow H^*(\Lambda^{i_0+1};\IZ), && *\leq2i_0+4.
\end{align}
It remains to obtain the first isomorphism also in degree $2i_0+3$. The long exact sequence of $(\Lambda^{i_0+1},\Lambda^{i_0})$ contains
\begin{equation}\label{equationexactsequencebelowreg}
\begin{aligned}
0\cong H^{2i_0+3}(\Lambda^{i_0+1},\Lambda^{i_0};\IZ)
&\longrightarrow H^{2i_0+3}(\Lambda^{i_0+1};\IZ)\\
&\longrightarrow H^{2i_0+3}(\Lambda^{i_0};\IZ)
\longrightarrow H^{2i_0+4}(\Lambda^{i_0+1},\Lambda^{i_0};\IZ).
\end{aligned}
\end{equation}
By \eqref{equationisomorphimfinalloopspacegroup},
\[
H^{2i_0+3}(\Lambda^{i_0+1};\IZ)\cong H^{2i_0+3}(\Lambda;\IZ)\cong\IZ.
\]
The long exact sequence of $(\Lambda^{i_0},\Lambda^{i_0-1})$ also gives
\[
\IZ\cong H^{2i_0+3}(\Lambda^{i_0},\Lambda^{i_0-1};\IZ)
\cong H^{2i_0+3}(\Lambda^{i_0};\IZ),
\]
and
\[
H^0(K_{i_0};\IZ)\cong H^{2i_0+4}(\Lambda^{i_0+1},\Lambda^{i_0};\IZ)\cong\IZ.
\]
Exactness of \eqref{equationexactsequencebelowreg} therefore implies that
\begin{align}\label{equationremainingiso}
H^{2i_0+3}(\Lambda^{i_0+1};\IZ)
\longrightarrow H^{2i_0+3}(\Lambda^{i_0};\IZ)
\end{align}
is an isomorphism. Together with \eqref{equationisomorphimfinalloopspacegroup}, the inclusion induces an isomorphism
\[
H^{2i_0+3}(\Lambda;\IZ)\longrightarrow H^{2i_0+3}(\Lambda^{i_0};\IZ).
\]
By \eqref{equationfundcupfreeloopspace}, therefore,
\begin{align}\label{equationsecondcolumniso}
\cdot\cup p^*([M]):H^{2i_0}(\Lambda^{i_0};\IZ)
\longrightarrow H^{2i_0+3}(\Lambda^{i_0};\IZ).
\end{align}

For the first column, use Proposition \ref{propmappingdegreecorespondence}, the fact that $M$ is diffeomorphic to $M$, and the isomorphism
\[
\cdot\cup[\IS^3]:H^0(\IS^3\times\IS^2;\IZ)
\longrightarrow H^3(\IS^3\times\IS^2;\IZ).
\]
Since $K_{i_0}$ is the first regular critical manifold, the diffeomorphism
\[
\Phi:K_{i_0}\longrightarrow T^1M,
\qquad c\longmapsto\frac{\dot c(0)}{\sqrt{E(c)}}
\]
shows that
\[
\cdot\cup\pi_M^*([M]):H^0(T^1M;\IZ)\longrightarrow H^3(T^1M;\IZ)
\]
is an isomorphism. Proposition \ref{propmappingdegreecorespondence} provides isomorphisms
\[
\begin{aligned}
k_1&:H^3(K_{i_0};\IZ)\longrightarrow H^{2i_0+3}(\Lambda^{i_0},\Lambda^{i_0-1};\IZ),\\
k_2&:H^0(K_{i_0};\IZ)\longrightarrow H^{2i_0}(\Lambda^{i_0},\Lambda^{i_0-1};\IZ),
\end{aligned}
\]
for which the diagram
\begin{center}
\begin{tikzpicture}[commutative diagrams/every diagram]
\matrix[matrix of math nodes, name=m, column sep=1cm, row sep=1cm, commutative diagrams/every cell] {
H^3(K_{i_0};\IZ) & H^{2i_0+3}(\Lambda^{i_0},\Lambda^{i_0-1};\IZ) & {} \\
H^0(K_{i_0};\IZ) & H^{2i_0}(\Lambda^{i_0},\Lambda^{i_0-1};\IZ)\\
};
\path[commutative diagrams/.cd, every arrow, every label]
(m-1-1) edge node{$k_1$} (m-1-2)
(m-2-1) edge node{$k_2$} (m-2-2)
(m-2-1) edge node{$\cup p^*([M])$} (m-1-1)
(m-2-2) edge node{$\cup p^*([M])$} (m-1-2);
\end{tikzpicture}
\end{center}
commutes up to sign. Since the left vertical map is an isomorphism, so is
\begin{align}\label{equationfirstcolumniso}
\cup p^*([M]):H^{2i_0}(\Lambda^{i_0},\Lambda^{i_0-1};\IZ)
\longrightarrow H^{2i_0+3}(\Lambda^{i_0},\Lambda^{i_0-1};\IZ).
\end{align}

Since both vertical maps in \eqref{diagramlastsublevelsetproof} are isomorphisms, exactness of its rows implies that
\[
H^{2i_0}(\Lambda^{i_0},\Lambda^{i_0-1};\IZ)
\longrightarrow H^{2i_0}(\Lambda^{i_0};\IZ)
\]
is an isomorphism. Hence $H^{2i_0}(\Lambda^{i_0-1};\IZ)=0$.
\end{proof}

By Theorem 9.1 of \cite{bokstedt2005spectral}, for $s\geq3$ the Poincar\'e series of $H^*_{\IS^1}(\Lambda\IS^s;\IZ/2\IZ)$ is
\[
P_{\IS^1}(\Lambda\IS^s,\IZ/2\IZ,t)
=\frac{1}{1-t^{2(s-1)}}\left(t^{s-1}+\frac{1+t^{2s-1}}{1-t^2}\right).
\]
From this we obtain that, for $i\geq 4$, the Betti numbers $b_i$ satisfy the recursion formula
\begin{align*}
b_i=\begin{cases} b_{i-2}+1,~~&i\equiv 2\mod 4,\\
b_{i-2},~~&i\equiv 0\mod 4,\\
b_{i+1}-2,~~&i\equiv 1\mod 2.
\end{cases}
\end{align*}
For $i>2$, this implies in particular that
\begin{align}\label{eq:firstequivalence}
\dim H^{2i+2} _{\IS^1}(\Lambda M;\IZ/2\IZ)> \dim H^{2i} _{\IS^1}(\Lambda M;\IZ/2\IZ)\Leftrightarrow i\equiv 0 \mod 2.
\end{align}
For the same range of indices,
\begin{align}\label{eq:secondequivalence}
\dim H^{2i+3} _{\IS^1}(\Lambda M;\IZ/2\IZ)>\dim H^{2i+1}  _{\IS^1}(\Lambda M;\IZ/2\IZ)\Leftrightarrow i\equiv 1 \mod 2.
\end{align}

\begin{prop}\label{propoddisotropyorder}
The index of the first regular critical manifold $K_{i_0}$ satisfies $\ind (K_{i_0})\equiv 2 \mod 4$.
\end{prop}

\begin{proof} In this proof, we abbreviate $\Lambda M$ by $\Lambda$. If $X$ is an $\IS^1$-space, we write $H^{*}(X)$ and $H^{*} _{\IS^1}(X)$ for the cohomology groups $H^{*}(X;\IZ/2\IZ)$ and $H^{*} _{\IS^1}(X;\IZ/2\IZ)$. Assume that $\ind(K_{i_0})=2i_0\equiv 0 \mod 4$, so $i_0\equiv 0\mod 2$. We use the Gysin sequence to show that $H^3 _{\IS^1}(K_{i_0};\IZ/2\IZ)$ vanishes, which, by Lemma \ref{lemregularleveltwotorsion}, contradicts the assumption $i_0\equiv 0\mod 2$. To this end, we show that the connecting map
\begin{align*}
H^{2i_0+2} _{\IS^1}(\Lambda^{i_0-1})\rightarrow H^{2i_0+3} _{\IS^1}(\Lambda^{i_0},\Lambda^{i_0-1})
\end{align*}
in the long exact sequence of the pair $(\Lambda^{i_0},\Lambda^{i_0-1})$ is zero and that the map
\begin{align*}
H^{2i_0+3} _{\IS^1}(\Lambda^{i_0})\rightarrow H^{2i_0+3} _{\IS^1}(\Lambda^{i_0-1})
\end{align*}
is an isomorphism.\\

By the index gap formula, there is no critical manifold of index $\ind(K_{i_0})+2$, so the inclusion
\begin{align}\label{indexgapformulainclusion}
\Lambda^{i_0}\rightarrow \Lambda
\end{align}
induces isomorphisms $H^{*}(\Lambda)\rightarrow H^{*}(\Lambda^{i_0})$ and $H^{*} _{\IS^1}(\Lambda)\rightarrow H^{*} _{\IS^1}(\Lambda^{i_0})$ for $*\leq 2i_0+2$. This gives the commutative diagram
\begin{equation}\label{isotoropybeweisgysindiagram2}
\resizebox{\textwidth}{!}{
\begin{tikzpicture}[commutative diagrams/every diagram]
\matrix[matrix of math nodes, name=m, column sep= 0.5cm, commutative diagrams/every cell, ampersand replacement =\&] {\& H^{2i_0-1} _{\IS^1}(\Lambda^{i_0},\Lambda^{i_0-1})\cong 0 \& \&  H^{2i_0+1} (\Lambda^{i_0},\Lambda^{i_0-1})\cong 0
\\H^{2i_0}(\Lambda^{i_0}) \& H^{2i_0-1} _{\IS^1}(\Lambda^{i_0})\& H^{2i_0+1} _{\IS^1}(\Lambda^{i_0}) \&H^{2i_0+1}(\Lambda^{i_0})\& H^{2i_0} _{\IS^1}(\Lambda^{i_0})\\
H^{2i_0}(\Lambda^{i_0-1}) \& H^{2i_0-1} _{\IS^1}(\Lambda^{i_0-1})\& H^{2i_0+1} _{\IS^1}(\Lambda^{i_0-1}) \& H^{2i_0+1}(\Lambda^{i_0-1})\& H^{2i_0} _{\IS^1}(\Lambda^{i_0-1})\\
};
\path[commutative diagrams/.cd, every arrow, every label]
(m-2-1) edge (m-2-2)
(m-2-2) edge (m-2-3)
(m-2-3) edge (m-2-4)
(m-2-4) edge (m-2-5)

(m-3-1) edge (m-3-2)
(m-3-2) edge (m-3-3)
(m-3-3) edge (m-3-4)
(m-3-4) edge (m-3-5)

(m-1-2) edge (m-2-2)
(m-1-4) edge (m-2-4)

(m-2-1) edge (m-3-1)
(m-2-2) edge (m-3-2)
(m-2-3) edge (m-3-3)
(m-2-4) edge (m-3-4)
(m-2-5) edge (m-3-5);
\end{tikzpicture}}
\end{equation}
whose rows are Gysin sequences and whose columns are pair sequences. We first show that the three middle vertical maps in (\ref{isotoropybeweisgysindiagram2}) are isomorphisms.
By assumption, $i_0-1\equiv 1 \mod 2$. The equivalence in (\ref{eq:secondequivalence}) therefore gives
\begin{align*}
\dim H^{2i_0+1} _{\IS^1}(\Lambda^{i_0})>\dim H^{2i_0-1} _{\IS^1}(\Lambda^{i_0}).
\end{align*}
Exactness of the rows in diagram \ref{isotoropybeweisgysindiagram2} implies that the maps $H^{2i_0+1} (\Lambda^{i_0})\rightarrow H^{2i_0} _{\IS^1}(\Lambda^{i_0})$ \\and $H^{2i_0} (\Lambda^{i_0})\rightarrow H^{2i_0-1} _{\IS^1}(\Lambda^{i_0})$ are zero. Lemma \ref{lastevensingularlevelzero} gives the isomorphism $H^{2i_0}(\Lambda^{i_0-1};\IZ)\cong 0$. Proposition \ref{propprilimcritcohomcomp} gives the isomorphisms $H^{2i_0-1}(\Lambda^{i_0-1};\IZ)\cong \IZ$ and $ H^{2i_0+1}(\Lambda^{i_0-1};\IZ)\cong \IZ$. Since $H^{*}(\Lambda^{i_0};\IZ)$ is finitely generated and $H^{2i_0+1}(\Lambda^{i_0};\IZ) $ is free, the universal coefficient theorem for cohomology gives
\begin{align}\label{eqzerolastsublevel}
H^{2i_0}(\Lambda^{i_0-1};\IZ/2\IZ)\cong H^{2i_0}(\Lambda^{i_0-1};\IZ)\otimes \IZ/2\IZ=0
\end{align}
and hence
\begin{align}\label{eqfreelastsublevel}
H^{2i_0-1}(\Lambda^{i_0-1})= H^{2i_0-1}(\Lambda^{i_0-1};\IZ/2\IZ)\cong H^{2i_0-1}(\Lambda^{i_0-1};\IZ)\otimes \IZ/2\IZ\cong \IZ/2\IZ.
\end{align}
The long exact sequence of the pair $(\Lambda^{i_0-1},\Lambda^{i_0-2})$ gives an isomorphism $H^{2i_0+1}(\Lambda^{i_0-1})\cong \IZ/2\IZ$. Exactness of the columns then implies that the vertical map
$$H^{2i_0+1} (\Lambda^{i_0})\rightarrow H^{2i_0+1} (\Lambda^{i_0-1})$$
is an isomorphism. Exactness of the rows and commutativity also imply that the connecting map in the Gysin sequence
$$H^{2i_0+1}(\Lambda^{i_0-1})\rightarrow H^{2i_0} _{\IS^1}(\Lambda^{i_0-1})$$
 is $0$. The map $H^{2i_0-1} _{\IS^1}(\Lambda^{i_0})\rightarrow H^{2i_0-1} _{\IS^1}(\Lambda^{i_0-1})$ is also an isomorphism. This follows from the long exact sequence of the pair $(\Lambda^{i_0},\Lambda^{i_0-1})$ and the Gysin sequence. Consider the diagram
\begin{equation}
\resizebox{\textwidth}{!}{
\begin{tikzpicture}[commutative diagrams/every diagram]
\matrix[matrix of math nodes, name=m, column sep= 0.5cm, commutative diagrams/every cell, ampersand replacement =\&] {\& H^{2i_0-3} _{\IS^1}(\Lambda^{i_0},\Lambda^{i_0-1})\cong 0 \& \&  H^{2i_0-1} (\Lambda^{i_0},\Lambda^{i_0-1})\cong 0
\\H^{2i_0-2}(\Lambda^{i_0}) \& H^{2i_0-3} _{\IS^1}(\Lambda^{i_0})\& H^{2i_0-1} _{\IS^1}(\Lambda^{i_0}) \&H^{2i_0-1}(\Lambda^{i_0})\& H^{2i_0-2} _{\IS^1}(\Lambda^{i_0})\\
H^{2i_0-2}(\Lambda^{i_0-1}) \& H^{2i_0-3} _{\IS^1}(\Lambda^{i_0-1})\& H^{2i_0-1} _{\IS^1}(\Lambda^{i_0-1}) \& H^{2i_0-1}(\Lambda^{i_0-1})\& H^{2i_0-2} _{\IS^1}(\Lambda^{i_0-1})\\
};
\path[commutative diagrams/.cd, every arrow, every label]
(m-2-1) edge (m-2-2)
(m-2-2) edge (m-2-3)
(m-2-3) edge (m-2-4)
(m-2-4) edge (m-2-5)

(m-3-1) edge (m-3-2)
(m-3-2) edge (m-3-3)
(m-3-3) edge (m-3-4)
(m-3-4) edge (m-3-5)

(m-1-2) edge (m-2-2)
(m-1-4) edge (m-2-4)

(m-2-1) edge (m-3-1)
(m-2-2) edge (m-3-2)
(m-2-3) edge (m-3-3)
(m-2-4) edge (m-3-4)
(m-2-5) edge (m-3-5);
\end{tikzpicture}}
\end{equation}
whose rows are Gysin sequences and whose columns are pair sequences. Since $H^{*} _{\IS^1}(\Lambda^{i_0},\Lambda^{i_0-1})\cong 0$ for $*\leq 2i_0=\ind (K_{i_0})$, exactness of the columns implies that the two leftmost vertical maps and the rightmost vertical map are isomorphisms. The inclusion $\Lambda^{i_0}\rightarrow \Lambda$ induces isomorphisms in cohomology for $*\leq 2i_0$, so we obtain an isomorphism $H^{2i_0-1}(\Lambda^{i_0})\cong \IZ/2\IZ$. Equation (\ref{eqfreelastsublevel}) gives an isomorphism $H^{2i_0-1}(\Lambda^{i_0-1})\cong \IZ/2\IZ$. Exactness of the columns then shows that the map $H^{2i_0-1} (\Lambda^{i_0})\rightarrow H^{2i_0-1} (\Lambda^{i_0-1})$ is an isomorphism. The five lemma therefore implies that the map
$$H^{2i_0-1} _{\IS^1}(\Lambda^{i_0})\rightarrow H^{2i_0-1} _{\IS^1}(\Lambda^{i_0-1})$$ 
is an isomorphism.\\

Thus, in diagram (\ref{isotoropybeweisgysindiagram2}), we can replace the leftmost and rightmost columns by columns of zeros to obtain the diagram
\begin{equation}\label{isotoropybeweisgysindiagram3}
\begin{tikzpicture}[commutative diagrams/every diagram]
\matrix[matrix of math nodes, name=m, column sep= 0.5cm, commutative diagrams/every cell] {&  0 & &  0
\\0 & H^{2i_0-1} _{\IS^1}(\Lambda^{i_0})& H^{2i_0+1} _{\IS^1}(\Lambda^{i_0}) &H^{2i_0+1}(\Lambda^{i_0})& 0\\
0 & H^{2i_0-1} _{\IS^1}(\Lambda^{i_0-1})& H^{2i_0+1} _{\IS^1}(\Lambda^{i_0-1}) &H^{2i_0+1}(\Lambda^{i_0-1})& 0\\
};
\path[commutative diagrams/.cd, every arrow, every label]
(m-2-1) edge (m-2-2)
(m-2-2) edge (m-2-3)
(m-2-3) edge (m-2-4)
(m-2-4) edge (m-2-5)

(m-3-1) edge (m-3-2)
(m-3-2) edge (m-3-3)
(m-3-3) edge (m-3-4)
(m-3-4) edge (m-3-5)

(m-1-2) edge (m-2-2)
(m-1-4) edge (m-2-4)

(m-2-1) edge (m-3-1)
(m-2-2) edge node {$\cong$}(m-3-2)
(m-2-3) edge (m-3-3)
(m-2-4) edge node {$\cong$}(m-3-4)
(m-2-5) edge (m-3-5);
\end{tikzpicture}
\end{equation}
which is commutative and has exact rows. By the five lemma, the map
\begin{align}\label{isotropyorderbeweisequation1}
H^{2i_0+1} _{\IS^1}(\Lambda^{i_0})\rightarrow H^{2i_0+1} _{\IS^1}(\Lambda^{i_0-1})
\end{align}
is an isomorphism. \\

Consider the diagram 
\begin{equation}\label{isotoropybeweisgysindiagram4}
\resizebox{\textwidth}{!}{
\begin{tikzpicture}[commutative diagrams/every diagram]
\matrix[matrix of math nodes, name=m, column sep= 0.4cm, commutative diagrams/every cell, ampersand replacement =\&] { \& \& 0\&0 \&
\\ \& H^{2i_0+1}(\Lambda^{i_0},\Lambda^{i_0-1})\cong 0\&H^{2i_0} _{\IS^1} (\Lambda^{i_0},\Lambda^{i_0-1})\& H^{2i_0+2} _{\IS^1}(\Lambda^{i_0},\Lambda^{i_0-1})\cong (\IZ/2\IZ)^2\\
{} \& H^{2i_0+1}(\Lambda^{i_0})\& H^{2i_0} _{\IS^1}(\Lambda^{i_0}) \& H^{2i_0+2} _{\IS^1}(\Lambda^{i_0})\& H^{2i_0+2}(\Lambda^{i_0})\& {}
\\
{} \&H^{2i_0+1}(\Lambda^{i_0-1})\& H^{2i_0} _{\IS^1}(\Lambda^{i_0-1}) \& H^{2i_0+2} _{\IS^1}(\Lambda^{i_0-1})\& 0\& {}\\
\& \&H^{2i_0+1} _{\IS^1}(\Lambda^{i_0},\Lambda^{i_0-1})=0\&H^{2i_0+3} _{\IS^1}(\Lambda^{i_0},\Lambda^{i_0-1})\\
};
\path[commutative diagrams/.cd, every arrow, every label]
(m-3-1) edge (m-3-2)
(m-3-2) edge node {}(m-3-3)
(m-3-3) edge (m-3-4)
(m-3-4) edge (m-3-5)
(m-3-5) edge node {} (m-3-6)

(m-4-1) edge (m-4-2)
(m-4-2) edge (m-4-3)
(m-4-3) edge (m-4-4)
(m-4-4) edge (m-4-5)

(m-1-3) edge (m-2-3)
(m-1-4) edge (m-2-4)

(m-2-2) edge (m-3-2)
(m-2-3) edge (m-3-3)
(m-2-4) edge (m-3-4)

(m-3-2) edge (m-4-2)
(m-3-3) edge (m-4-3)
(m-3-4) edge (m-4-4)
(m-3-5) edge (m-4-5)

(m-4-3) edge (m-5-3)
(m-4-4) edge (m-5-4)
;
\end{tikzpicture}}
\end{equation}
whose rows are the Gysin sequences of the bundles $\Lambda^{i_0}\rightarrow  \Lambda^{i_0} _{\IS^1}$ and $\Lambda^{i_0-1} \rightarrow \Lambda _{\IS^1} ^{i_0-1}$ and whose columns come from the exact sequence of the pair $(\Lambda^{i_0},\Lambda^{i_0-1})$. By the index gap formula, the inclusion \ref{indexgapformulainclusion} induces isomorphisms in cohomology in degrees at most $i_0+2$. Since $i_0\equiv 0 \mod 2$, it follows from (\ref{eq:firstequivalence}) that
\begin{align*}
\dim H^{2i_0+2}( \Lambda^{i_0})>\dim H^{2i_0}(\Lambda^{i_0}).
\end{align*}
The dimensions of the cohomology groups imply that the two connecting maps in the Gysin sequence, $H^{2i_0+1}(\Lambda^{i_0})\rightarrow H^{2i_0} _{\IS^1}(\Lambda^{i_0})$ and $H^{2i_0+2} (\Lambda^{i_0})\rightarrow H^{2i_0+1} _{\IS^1}(\Lambda^{i_0})$, are zero. Together with the isomorphism $H^{2i_0+1}(\Lambda^{i_0})\rightarrow H^{2i_0+1}(\Lambda^{i_0-1})$, this shows that
\begin{align}
H^{2i_0} _{\IS^1}(\Lambda^{i_0-1})\rightarrow H^{2i_0+2} _{\IS^1}(\Lambda^{i_0-1})
\end{align}
is an isomorphism, and hence $H^{2i_0+2} _{\IS^1}(\Lambda^{i_0})\rightarrow H^{2i_0+2} _{\IS^1}(\Lambda^{i_0-1})$ is surjective. It follows that the connecting map $H^{2i_0+2} _{\IS^1}(\Lambda^{i_0-1})\rightarrow H^{2i_0+3} _{\IS^1}(\Lambda^{i_0},\Lambda^{i_0-1})$ is zero. This proves the first claim. \\

It remains to show that the map $H^{2i_0+3} _{\IS^1}(\Lambda^{i_0})\rightarrow H^{2i_0+3} _{\IS^1}(\Lambda^{i_0-1})$ is an isomorphism. For this, consider the diagram
\begin{equation}\label{isotoropybeweisgysindiagram5}
\resizebox{\textwidth}{!}{
\begin{tikzpicture}[commutative diagrams/every diagram]
\matrix[matrix of math nodes, name=m, column sep= 0.4cm, commutative diagrams/every cell, ampersand replacement =\&] {H^{2i_0+1} _{\IS^1}(\Lambda^{i_0},\Lambda^{i_0-1})\& H^{2i_0+3} _{\IS^1}(\Lambda^{i_0},\Lambda^{i_0-1})\&H^{2i_0+3} (\Lambda^{i_0},\Lambda^{i_0-1})\&\\
H^{2i_0+1} _{\IS^1}(\Lambda^{i_0})\& H^{2i_0+3} _{\IS^1}(\Lambda^{i_0})\&H^{2i_0+3}(\Lambda^{i_0})\&H^{2i_0+2} _{\IS^1}(\Lambda^{i_0})\& H^{2i_0+4} _{\IS^1}(\Lambda^{i_0})\& 0\\
H^{2i_0+1} _{\IS^1}(\Lambda^{i_0-1})\& H^{2i_0+3} _{\IS^1}(\Lambda^{i_0-1})\&H^{2i_0+3}(\Lambda^{i_0-1})\cong 0 \&H^{2i_0+2} _{\IS^1}(\Lambda^{i_0-1})\&H^{2i_0+4} _{\IS^1}(\Lambda^{i_0-1})\& 0.
\\};
\path[commutative diagrams/.cd, every arrow, every label]
(m-1-1) edge (m-2-1)
(m-1-2) edge (m-2-2)
(m-1-3) edge (m-2-3)

(m-2-1) edge (m-2-2)
(m-2-2) edge (m-2-3)
(m-2-3) edge (m-2-4)
(m-2-4) edge (m-2-5)
(m-2-5) edge (m-2-6)

(m-3-1) edge (m-3-2)
(m-3-2) edge (m-3-3)
(m-3-3) edge (m-3-4)
(m-3-4) edge (m-3-5)
(m-3-5) edge (m-3-6)

(m-2-1) edge (m-3-1)
(m-2-2) edge (m-3-2)
(m-2-3) edge (m-3-3)
(m-2-4) edge (m-3-4)
(m-2-5) edge (m-3-5);
\end{tikzpicture}}
\end{equation}
whose columns are pair sequences and whose rows, as in the preceding diagrams, are the Gysin sequences of the bundles $\Lambda^{i_0}\rightarrow \Lambda_{\IS^1} ^{i_0}$ and $\Lambda^{i_0-1}\rightarrow \Lambda_{\IS^1} ^{i_0-1}$. We first show that the induced maps $H^{2i_0+3} _{\IS^1}(\Lambda)\rightarrow H^{2i_0+3}(\Lambda^{i_0})$ and $H^{2i_0+3}(\Lambda)\rightarrow H^{2i_0+3} _{\IS^1}(\Lambda^{i_0})$ are isomorphisms and then use the assumption $i_0\equiv 0 \mod 2$. The long exact sequence of the pair $(\Lambda^{i_0},\Lambda^{i_0-1})$ gives
\begin{align*}
H^{2i_0+3}(\Lambda^{i_0})\cong H^{2i_0+3}(\Lambda^{i_0},\Lambda^{i_0-1})\cong \IZ/2\IZ.
\end{align*}
By the index gap formula, the critical manifold $K_{i_0+1}$ has index $\ind(K_{i_0})+4$, so the inclusion $\Lambda^{i_0+1}\rightarrow \Lambda$ induces an isomorphism $H^{2i_0+3}  (\Lambda^{i_0+1})\cong H^{2i_0+3}(\Lambda)\cong \IZ/2\IZ$. Together with $H^{2i_0+3}(\Lambda^{i_0+1},\Lambda^{i_0})=0$, this implies that the map
\begin{align}\label{indexplus3iso}
H^{2i_0+3}(\Lambda^{i_0+1})\rightarrow H^{2i_0+3}(\Lambda^{i_0})
\end{align}
in the pair sequence is an isomorphism. The five lemma, applied to the commutative diagram
\begin{equation}\label{isotoropybeweisgysindiagram6}
\resizebox{\textwidth}{!}{
\begin{tikzpicture}[commutative diagrams/every diagram]
\matrix[matrix of math nodes, name=m, column sep= 1cm, commutative diagrams/every cell, ampersand replacement =\&] {H^{2i_0+2}(\Lambda^{i_0+1})\& H^{2i_0+1} _{\IS^1}(\Lambda^{i_0+1})\& H^{2i_0+3} _{\IS^1}(\Lambda^{i_0+1})\&H^{2i_0+3}(\Lambda^{i_0+1})\& H^{2i_0+2} _{\IS^1}(\Lambda^{i_0+1})\\
H^{2i_0+2}(\Lambda^{i_0})\& H^{2i_0+1} _{\IS^1}(\Lambda^{i_0})\& H^{2i_0+3} _{\IS^1}(\Lambda^{i_0})\&H^{2i_0+3}(\Lambda^{i_0})\& H^{2i_0+2} _{\IS^1}(\Lambda^{i_0})\\
};
\path[commutative diagrams/.cd, every arrow, every label]
(m-1-1) edge (m-1-2)
(m-1-2) edge (m-1-3)
(m-1-3) edge (m-1-4)
(m-1-4) edge (m-1-5)

(m-2-1) edge (m-2-2)
(m-2-2) edge (m-2-3)
(m-2-3) edge (m-2-4)
(m-2-4) edge (m-2-5)

(m-1-1) edge node {$\cong$}(m-2-1)
(m-1-2) edge node {$\cong$}(m-2-2)
(m-1-3) edge (m-2-3)
(m-1-4) edge node {$\cong$}(m-2-4)
(m-1-5) edge node {$\cong$}(m-2-5);
\end{tikzpicture}}
\end{equation}
whose two leftmost and two rightmost vertical maps are isomorphisms, shows that the map
\begin{align}\label{indexplus3eqiso}
H^{2i_0+3} _{\IS^1}(\Lambda^{i_0+1})\cong H^{2i_0+3} _{\IS^1}(\Lambda)\rightarrow H^{2i_0+3} _{\IS^1}(\Lambda^{i_0})
\end{align}
is an isomorphism. Combining the isomorphisms \ref{indexplus3eqiso}, the isomorphism (\ref{indexplus3iso}), and the isomorphisms
$$H^{2i_0+2} (\Lambda)\rightarrow H^{2i_0+2} (\Lambda^{i_0})$$ and 
$$H^{2i_0+1} (\Lambda)\rightarrow H^{2i_0+1} (\Lambda^{i_0})$$ with \ref{eq:secondequivalence}, we deduce from $i_0\equiv 0 \mod 2$ that
\begin{align*}
\dim H^{2i_0+3} _{\IS^1}(\Lambda^{i_0})=\dim H^{2i_0+1} _{\IS^1}(\Lambda^{i_0}).
\end{align*}
Since the connecting map $H^{2i_0+2} (\Lambda^{i_0})\rightarrow H^{2i_0+1} _{\IS^1}(\Lambda^{i_0})$ is zero, this equality implies, by exactness, that the map
\begin{align}
H^{2i_0+1} _{\IS^1}(\Lambda^{i_0})\rightarrow H^{2i_0+3} _{\IS^1}(\Lambda^{i_0})
\end{align}
in the Gysin sequence is an isomorphism. Using this isomorphism and the one in (\ref{isotropyorderbeweisequation1}), we can write diagram (\ref{isotoropybeweisgysindiagram5}) as
\begin{equation}\label{isotoropybeweisgysindiagram7}\resizebox{\textwidth}{!}{
\begin{tikzpicture}[commutative diagrams/every diagram]
\matrix[matrix of math nodes, name=m, column sep= 0.3cm, commutative diagrams/every cell, ampersand replacement =\&] {{} \& H^{2i_0+1} _{\IS^1}(\Lambda^{i_0},\Lambda^{i_0-1})\cong 0 \& H^{2i_0+3} _{\IS^1}(\Lambda^{i_0},\Lambda^{i_0-1})\&H^{2i_0+3} (\Lambda^{i_0},\Lambda^{i_0-1})\&\\
0\& H^{2i_0+1} _{\IS^1}(\Lambda^{i_0})\& H^{2i_0+3} _{\IS^1}(\Lambda^{i_0})\&H^{2i_0+3}(\Lambda^{i_0})\&H^{2i_0+2} _{\IS^1}(\Lambda^{i_0})\&H^{2i_0+4} _{\IS^1}(\Lambda^{i_0})\& 0\\
0 \& H^{2i_0+1} _{\IS^1}(\Lambda^{i_0-1})\& H^{2i_0+3} _{\IS^1}(\Lambda^{i_0-1})\& 0 \& H^{2i_0+2} _{\IS^1}(\Lambda^{i_0-1})\&H^{2i_0+4} _{\IS^1}(\Lambda^{i_0-1})\& 0.
\\};
\path[commutative diagrams/.cd, every arrow, every label]
(m-1-4) edge (m-2-4)
(m-1-2) edge (m-2-2)
(m-1-3) edge (m-2-3)

(m-2-1) edge (m-2-2)
(m-2-2) edge node{$\cong$} (m-2-3)
(m-2-3) edge node {$0$}(m-2-4)
(m-2-4) edge (m-2-5)
(m-2-5) edge (m-2-6)
(m-2-6) edge (m-2-7)

(m-3-1) edge (m-3-2)
(m-3-2) edge (m-3-3)
(m-3-3) edge (m-3-4)
(m-3-4) edge (m-3-5)
(m-3-5) edge (m-3-6)
(m-3-6) edge (m-3-7)

(m-2-2) edge node {$\cong$}(m-3-2)
(m-2-3) edge (m-3-3)
(m-2-4) edge (m-3-4)
(m-2-5) edge (m-3-5)
(m-2-6) edge (m-3-6);
\end{tikzpicture}}
\end{equation}
Commutativity of the diagram and exactness of its rows and columns imply that
\begin{align}
H^{2i_0+3} _{\IS^1}(\Lambda^{i_0}) \rightarrow H^{2i_0+3} _{\IS^1}(\Lambda^{i_0-1})
\end{align}
is an isomorphism. Together with exactness of the columns in (\ref{isotoropybeweisgysindiagram7}), this implies that
$$ H^{2i_0+3} _{\IS^1}(\Lambda^{i_0},\Lambda^{i_0-1})\cong H^3(K_{i_0})\cong 0,$$
which, by Lemma \ref{lemregularleveltwotorsion}, contradicts the assumption $i_0\equiv 0 \mod 2$.
\end{proof}

\begin{theorem}\label{coroddisoorder}
We have $q_i\equiv 1 \mod 2$ for $i=1,...,k$.
\end{theorem}

\begin{proof}
Suppose that some $j\in\{1,...,k\}$ satisfies
\begin{align}\label{corproofisotropycongruence1}
q_j\equiv 0 \mod 2.
\end{align} 
By Proposition \ref{propoddisotropyorder}, we have $2i_0=\ind(K_{i_0})\equiv 2 \mod 4$, which gives
\begin{align*}
\sum_{i=1} ^k (q_i-1)
&= \#\{K\subset \Lambda M,~ K\textup{ is critical manifold of}~ E~\textup{ with}~ \ind (K)<\ind(K_{i_0}) \}\\
&= i_0-1\equiv 0 \mod 2.
\end{align*}
Together with (\ref{corproofisotropycongruence1}), this implies that there is $j_1\in \{1,...,k\}$ with $j\neq j_1$ and $q_{j_1}\equiv 0\mod 2$. This contradicts Proposition \ref{propatmostoneevenisoorder} and completes the proof.
\end{proof}




\section{The case that the geodesic flow has exactly one singular orbit type}

We now exclude the case that the isotropy set of the geodesic flow has exactly one connected component. First we compute the low-degree groups $H_{\IS^1}^{*}(\Lambda M,M;\IZ)$ using the round metric on $\IS^3$; these groups are torsion-free below degree $10$. We then repeat the computation for a metric with nonempty connected isotropy set and obtain $2$-torsion already in degree $6$, giving a contradiction.

\begin{lemma}\label{lemequcohomgroups}
For $*\leq 10$ the integral $\IS^1$-equivariant cohomology groups are
\begin{align*}
H^{*}_{\IS^1}(\Lambda M,M;\IZ)\cong
\begin{cases}
\IZ,&*=2,\\
\IZ^2,&*=4,6,8,\\
\IZ^2\oplus \IZ/4\IZ,&*=10,\\
0,&\textup{elsewhere}.
\end{cases}
\end{align*}
In particular, the first torsion occurs in degree $10$, and all odd-degree groups vanish.
\end{lemma}

\begin{proof}
We use Morse theory for the round metric. Then $K_1\cong T^1M$ consists of the prime geodesics, $K_i=K_1^i$ for $i>1$, and
\[
\ind(K_i)=2+4(i-1).
\]
By Theorem \ref{thmcohomologycritmanifolds}, the groups needed below are
\begin{align*}
H^{*}_{\IS^1}(K_1;\IZ)&\cong
\begin{cases}
\IZ,&*=0,\\
\IZ^2,&*=2,\\
\IZ,&*=4,\\
0,&\textup{elsewhere},
\end{cases}\\[1mm]
H^{*}_{\IS^1}(K_2;\IZ)&\cong
\begin{cases}
\IZ,&*=0,\\
\IZ^2,&*=2,\\
\IZ\oplus\IZ/4\IZ,&*=4,\\
\IZ/4\IZ\oplus\IZ/4\IZ,&*>4,\ *\equiv0\mod 2,\\
0,&\textup{elsewhere}.
\end{cases}
\end{align*}
Hence the equivariant Thom isomorphism (\ref{subsevelthomisoequivariant}) gives
\begin{align}
H^{*}_{\IS^1}(\Lambda^1,\Lambda^0;\IZ)&\cong
\begin{cases}
\IZ,&*=2,\\
\IZ^2,&*=4,\\
\IZ,&*=6,\\
0,&\textup{elsewhere},
\end{cases}
\label{proofrelcohfirst}\\
H^{*}_{\IS^1}(\Lambda^2,\Lambda^1;\IZ)&\cong
\begin{cases}
\IZ,&*=6,\\
\IZ^2,&*=8,\\
\IZ\oplus\IZ/4\IZ,&*=10,\\
\IZ/4\IZ\oplus\IZ/4\IZ,&*>10,\ *\equiv0\mod 2,\\
0,&\textup{elsewhere}.
\end{cases}
\label{proofrelcohsecond}
\end{align}
The long exact sequences of $(\Lambda^1,M)$ and $(\Lambda^2,\Lambda^1,M)$ therefore give
\begin{align*}
H^{2}_{\IS^1}(\Lambda^1,M;\IZ)&\cong\IZ,\\
H^{4}_{\IS^1}(\Lambda^1,M;\IZ)&\cong\IZ^2,\\
H^{6}_{\IS^1}(\Lambda^2,M;\IZ)&\cong\IZ^2,\\
H^{8}_{\IS^1}(\Lambda^2,M;\IZ)&\cong\IZ^2,\\
H^{10}_{\IS^1}(\Lambda^2,M;\IZ)&\cong\IZ\oplus\IZ/4\IZ.
\end{align*}
For $*\leq\ind(K_{i+1})-2$, the inclusion $\Lambda^i\hookrightarrow\Lambda M$ induces an isomorphism
\[
H^{*}_{\IS^1}(\Lambda M,M;\IZ)\longrightarrow H^{*}_{\IS^1}(\Lambda^i,M;\IZ).
\]
For the round sphere, Corollary 3.2 of \cite{radeschi2017berger} gives
\[
H^{2*+1}_{\IS^1}(\Lambda M,M;\IZ)=0.
\]
Thus only degree $10$ remains. 

For computing the ordinary cohomology, observe that it holds $K_i\cong T^1M\cong\IS^3\times\IS^2$. The Thom isomorphism gives
\begin{align*}
H^{*}(\Lambda^3,\Lambda^2;\IZ)&\cong
\begin{cases}\IZ,&*=10,12,13,15,\\0,&\textup{elsewhere},\end{cases}\\
H^{*}(\Lambda^1,\Lambda^0;\IZ)&\cong
\begin{cases}\IZ,&*=2,4,5,7,\\0,&\textup{elsewhere},\end{cases}\\
H^{*}(\Lambda^2,\Lambda^1;\IZ)&\cong
\begin{cases}\IZ,&*=6,8,9,11,\\0,&\textup{elsewhere}.
\end{cases}
\end{align*}
From the triple $(\Lambda^2,\Lambda^1,M)$ we obtain, in particular,
\[
H^9(\Lambda^2,M;\IZ)\cong\IZ,
\qquad
H^{10}(\Lambda^2,M;\IZ)=0.
\]
The degree-$9$ part of the exact sequence of $(\Lambda^3,\Lambda^2,M)$ is
\begin{center}
\begin{tikzpicture}[commutative diagrams/every diagram]
\matrix[matrix of math nodes, name=m, column sep=0.8cm, commutative diagrams/every cell, ampersand replacement=\&] {
0 \& H^9(\Lambda^3,M;\IZ) \& H^9(\Lambda^2,M;\IZ)\cong\IZ \& H^{10}(\Lambda^3,\Lambda^2;\IZ)\cong\IZ \\
};
\path[commutative diagrams/.cd, every arrow, every label]
(m-1-1) edge (m-1-2)
(m-1-2) edge (m-1-3)
(m-1-3) edge (m-1-4);
\end{tikzpicture}
\end{center}
As in the original argument, Proposition \ref{totalcohomiso} gives
$H^9(\Lambda^3,M;\IZ)\cong H^9(\Lambda M,M;\IZ)\cong\IZ$; hence the inclusion induces an isomorphism
\begin{equation}\label{inclusiondeg9iso}
H^9(\Lambda^3,M;\IZ)\xrightarrow{\cong}H^9(\Lambda^2,M;\IZ).
\end{equation}
For $i=2,3$ write $\delta_i$ for the connecting map of the relative Gysin sequence of the homotopy quotient of the sublevel sets. Combining the relative Gysin sequence of the sublevel sets and the triple sequence gives the commutative diagram
\begin{center}
\begin{tikzpicture}[commutative diagrams/every diagram]
\matrix[matrix of math nodes, name=m, column sep=0.9cm, row sep=0.7cm,
commutative diagrams/every cell, ampersand replacement=\&] {
0 \& H^9(\Lambda^3,M;\IZ) \& H^8_{\IS^1}(\Lambda^3,M;\IZ) \& H^{10}_{\IS^1}(\Lambda^3,M;\IZ) \& H^{10}(\Lambda^3,M;\IZ) \& 0 \\
0 \& H^9(\Lambda^2,M;\IZ) \& H^8_{\IS^1}(\Lambda^2,M;\IZ) \& H^{10}_{\IS^1}(\Lambda^2,M;\IZ) \& H^{10}(\Lambda^2,M;\IZ) \& 0 \\
};
\path[commutative diagrams/.cd, every arrow, every label]
(m-1-1) edge (m-1-2)
(m-1-2) edge node {$\delta_3$} (m-1-3)
(m-1-3) edge node {$\cup e(\Lambda^3)$} (m-1-4)
(m-1-4) edge node {} (m-1-5)
(m-1-5) edge (m-1-6)

(m-2-1) edge (m-2-2)
(m-2-2) edge node {$\delta_2$} (m-2-3)
(m-2-3) edge node {$\cup e(\Lambda^2)$} (m-2-4)
(m-2-4) edge node {} (m-2-5)
(m-2-5) edge (m-2-6)

(m-1-2) edge (m-2-2)
(m-1-3) edge (m-2-3)
(m-1-4) edge (m-2-4)
(m-1-5) edge (m-2-5);
\end{tikzpicture}
\end{center}
Here $e(\Lambda^i)$ denotes the Euler class of
$\Lambda^i\times E\IS^1\to\Lambda^i\times_{\IS^1}E\IS^1$ for $i=2,3$. Using \eqref{inclusiondeg9iso} and the exact sequences of the triple, the vertical map in equivariant degree $8$ is also an isomorphism. Hence
\[
\coker(\delta_3)\cong\coker(\delta_2).
\]
Since $H^{10}(\Lambda^2,M;\IZ)=0$, exactness of the lower row gives
\[
\coker(\delta_2)\cong H^{10}_{\IS^1}(\Lambda^2,M;\IZ)
\cong\IZ\oplus\IZ/4\IZ.
\]
Finally $H^{10}(\Lambda^3,M;\IZ)\cong\IZ$ is free, so the exact sequence
\[
0\longrightarrow \coker(\delta_3)
\longrightarrow H^{10}_{\IS^1}(\Lambda^3,M;\IZ)
\longrightarrow H^{10}(\Lambda^3,M;\IZ)
\longrightarrow0
\]
splits. Therefore
\[
H^{10}_{\IS^1}(\Lambda^3,M;\IZ)
\cong(\IZ\oplus\IZ/4\IZ)\oplus\IZ,
\]
which proves the lemma.
\end{proof}

\begin{lemma}\label{lemmadegfundcorrespondence}
Recall that $i_0\in\IN$ was minimal with the property that $K_{i_0}$ is a regular critical manifold. Assume that $i_0>2$ and that the first two critical manifolds are singular. Then
\[
\#\pi_1(K_2)\mid\deg p|_{K_1},
\qquad
\deg p|_{K_2}\mid\#\pi_1(K_1).
\]
In particular, if $K_2=K_1^2$, then
\[
\deg(K_1)=\pm\deg p|_{K_2}
=\pm\#\pi_1(K_1)=\pm\#\pi_1(K_2),
\]
and in this case $H^6(\Lambda^2;\IZ)=0$.
\end{lemma}

\begin{proof}
The proof is the same diagram chase as in the original version, using index monotonicity and the fact that there is at most one critical manifold at each critical energy. Let $m\in M$. As before,
\[
H^{*}(\Lambda M;\IZ)
\cong H^{*}(\Omega_mM;\IZ)\otimes_{\IZ}H^{*}(M;\IZ)
\cong H^{*}(\Omega\IS^3;\IZ)\otimes_{\IZ}H^{*}(\IS^3;\IZ),
\]
where $\Omega_m$ denotes the loop space of $M$ over the point $m\in M$.
Thus, for a generator $[M]\in H^3(M;\IZ)$,
\[
\cdot\cup p^*([M]):H^{2*}(\Lambda M;\IZ)\xrightarrow{\cong}H^{2*+3}(\Lambda M;\IZ).
\]
Since $\Lambda^0\cong M$ and Proposition \ref{propbasicisos} identifies the relevant degree-$3$ groups, the long exact sequence of $(\Lambda^1,\Lambda^0)$ gives that
$H^3(\Lambda^1;\IZ)\to H^3(\Lambda^0;\IZ)$ is an isomorphism and hence
\[
H^4(\Lambda^1;\IZ)\cong\IZ/\#\pi_1(K_1)\IZ.
\]

Proposition \ref{propmappingdegreecorespondence} identifies cupping product with the class $p^{*}M$ with multiplication by the mapping degree. The corresponding naturality square is
\begin{center}
\begin{tikzpicture}[commutative diagrams/every diagram]
\matrix[matrix of math nodes, name=m, column sep=1.1cm, row sep=0.7cm,
commutative diagrams/every cell, ampersand replacement=\&] {
H^2(\Lambda^1,\Lambda^0;\IZ)\cong\IZ \& H^2(\Lambda^1;\IZ)\cong\IZ \\
H^5(\Lambda^1,\Lambda^0;\IZ)\cong\IZ \& H^5(\Lambda^1;\IZ)\cong\IZ \\
};
\path[commutative diagrams/.cd, every arrow, every label]
(m-1-1) edge (m-1-2)
(m-2-1) edge (m-2-2)
(m-1-1) edge[right] node {$\cdot\deg(p|_{K_1})$} (m-2-1)
(m-1-2) edge[right] node {$\cup p^*([M])$} (m-2-2);
\end{tikzpicture}
\end{center}
It commutes up to sign; since the horizontal maps are isomorphisms, cup product on $H^2(\Lambda^1;\IZ)$ is multiplication by $\deg(p|_{K_1})$, up to sign. Now consider the commutative diagram of the long exact sequences for the pair $(\Lambda^2,\Lambda^1)$:
\begin{center}
\begin{tikzpicture}[commutative diagrams/every diagram]
\matrix[matrix of math nodes, name=m, column sep=0.8cm, row sep=0.65cm,
commutative diagrams/every cell, ampersand replacement=\&] {
0 \& H^5(\Lambda^2;\IZ)\cong\IZ \& H^5(\Lambda^1;\IZ)\cong\IZ \& H^6(\Lambda^2,\Lambda^1;\IZ)\cong\IZ/\#\pi_1(K_2)\IZ \\
0 \& H^2(\Lambda^2;\IZ) \& H^2(\Lambda^1;\IZ)\cong\IZ \& 0 \\
};
\path[commutative diagrams/.cd, every arrow, every label]
(m-1-1) edge (m-1-2)
(m-1-2) edge node {$j_1^5$} (m-1-3)
(m-1-3) edge (m-1-4)
(m-2-1) edge (m-2-2)
(m-2-2) edge node {$j_1^2$} (m-2-3)
(m-2-3) edge (m-2-4)
(m-2-2) edge[right] node {$\cup p^*([M])$} (m-1-2)
(m-2-3) edge[right] node {$\cdot\deg(p|_{K_1})$} (m-1-3);
\end{tikzpicture}
\end{center}
Hence $j_1^5$ is multiplication by $\deg(p|_{K_1})$, and the original diagram chase gives
\[
\deg(p|_{K_1})\mid\#\pi_1(K_2).
\]

Likewise, Proposition \ref{propmappingdegreecorespondence} identifies
$H^4(\Lambda^2,\Lambda^1;\IZ)\to H^7(\Lambda^2,\Lambda^1;\IZ)$
with multiplication by $\deg(p|_{K_2})$. Comparing this with the exact fragment
\[
0\longrightarrow H^4(\Lambda^2,\Lambda^1;\IZ)\cong\IZ
\xrightarrow{j_2^4}H^4(\Lambda^2;\IZ)\cong\IZ
\longrightarrow H^4(\Lambda^1;\IZ)\cong\IZ/\#\pi_1(K_1)\IZ
\longrightarrow0
\]
shows, exactly as in the original diagram, that $j_2^4$ is multiplication by $\deg(p|_{K_2})$ and hence
\[
\#\pi_1(K_1)\mid\deg(p|_{K_2}).
\]

If $K_2=K_1^2$, then $K_1$ and $K_2$ are diffeomorphic, and these relations imply
\[
\deg(K_1)=\pm\deg p|_{K_2}
=\pm\#\pi_1(K_1)=\pm\#\pi_1(K_2).
\]
The relevant part of the exact sequence is
\[
0\longrightarrow H^5(\Lambda^2;\IZ)\cong\IZ
\xrightarrow{\cdot\#\pi_1(K_2)}H^5(\Lambda^1;\IZ)\cong\IZ
\xrightarrow{\delta^5_{\Lambda^2,\Lambda^1}}
H^6(\Lambda^2,\Lambda^1;\IZ)\cong\IZ/\#\pi_1(K_2)\IZ
\xrightarrow{j_2^6}H^6(\Lambda^2;\IZ)\longrightarrow0.
\]
Thus $\delta^5_{\Lambda^2,\Lambda^1}$ is surjective, so $j_2^6=0$ and $H^6(\Lambda^2;\IZ)=0$.
\end{proof}

This allows us to prove Theorem \ref{thmoneorbittype}.

\begin{proof}[Proof of Theorem \ref{thmoneorbittype}]
Assume that the isotropy set of the geodesic flow is a single connected $3$-manifold $C$. Then $\Phi(K)=C$ for every singular critical manifold $K$. The lowest critical manifold $K_1$ consists of prime geodesics and is diffeomorphic to $C$. By Theorem \ref{coroddisoorder}, $K_2$ is singular as well, so
\[
\Phi(K_2)=\Phi(K_1),
\qquad
K_2=K_1^2.
\]
We compute $H^*_{\IS^1}(\Lambda M,M;\IZ)$ for this metric and compare it with Lemma \ref{lemequcohomgroups}. By Theorem \ref{thmcohomologycritmanifolds},
\begin{align*}
H^{*}_{\IS^1}(K_1;\IZ)&\cong
\begin{cases}
\IZ,&*=0,2,\\
0,&\textup{elsewhere},
\end{cases}\\[1mm]
H^{*}_{\IS^1}(K_2;\IZ)&\cong
\begin{cases}
\IZ,&*=0,\\
\IZ\oplus\IZ/2\IZ,&*=2,\\
\IZ/2\IZ\oplus\IZ/2\IZ,&*>2,\ *\equiv0\mod 2,\\
0,&\textup{elsewhere}.
\end{cases}
\end{align*}
The first three indices are
\[
\ind(K_1)=2,
\qquad \ind(K_2)=4,
\qquad \ind(K_3)=6.
\]
Hence the equivariant Thom isomorphism (\ref{subsevelthomisoequivariant}) gives
\begin{align*}
H^{*}_{\IS^1}(\Lambda^1,M;\IZ)&\cong
\begin{cases}
\IZ,&*=2,4,\\
0,&\textup{elsewhere},
\end{cases}\\
H^{*}_{\IS^1}(\Lambda^2,\Lambda^1;\IZ)&\cong
\begin{cases}
\IZ,&*=4,\\
\IZ\oplus\IZ/2\IZ,&*=6,\\
\IZ/2\IZ\oplus\IZ/2\IZ,&*>6,\ *\equiv0\mod 2,\\
0,&\textup{elsewhere}.
\end{cases}
\end{align*}
The long exact sequence of $(\Lambda^2,\Lambda^1,M)$ now yields
\begin{align*}
H^2_{\IS^1}(\Lambda^1,M;\IZ)&\cong\IZ,\\
H^4_{\IS^1}(\Lambda^2,M;\IZ)&\cong\IZ^2,\\
H^6_{\IS^1}(\Lambda^2,M;\IZ)&\cong\IZ\oplus\IZ/2\IZ,\\
H^{2*+1}_{\IS^1}(\Lambda^2,M;\IZ)&=0.
\end{align*}
The triple $(\Lambda^3,\Lambda^2,M)$ also gives
\[
H^3_{\IS^1}(\Lambda^3,M;\IZ)=0,
\qquad
H^4_{\IS^1}(\Lambda^3,M;\IZ)\cong\IZ^2.
\]
For $*\leq\ind(K_{i+1})-2$, the inclusion $\Lambda^i\hookrightarrow\Lambda M$ induces an isomorphism
\[
H^{*}_{\IS^1}(\Lambda M,M;\IZ)
\longrightarrow H^{*}_{\IS^1}(\Lambda^i,M;\IZ).
\]
In particular,
\[
H^6_{\IS^1}(\Lambda M,M;\IZ)
\cong H^6_{\IS^1}(\Lambda^3,M;\IZ).
\]

It remains to compute the latter group. Let $\delta_i$ be the connecting map in the relative Gysin sequence of the sublevel sets. As in the preceding proof the relative Gysin sequence and the long exact triple sequence give the commutative diagram with exact rows and columns
\begin{center}
\begin{tikzpicture}[commutative diagrams/every diagram]
\matrix[matrix of math nodes, name=m, column sep=0.9cm, row sep=0.7cm,
commutative diagrams/every cell, ampersand replacement=\&] {
0 \& H^5(\Lambda^3,M;\IZ) \& H^4_{\IS^1}(\Lambda^3,M;\IZ) \& H^6_{\IS^1}(\Lambda^3,M;\IZ) \& H^6(\Lambda^3,M;\IZ) \& 0 \\
0 \& H^5(\Lambda^2,M;\IZ) \& H^4_{\IS^1}(\Lambda^2,M;\IZ) \& H^6_{\IS^1}(\Lambda^2,M;\IZ) \& H^6(\Lambda^2,M;\IZ) \& 0 \\
};
\path[commutative diagrams/.cd, every arrow, every label]
(m-1-1) edge (m-1-2)
(m-1-2) edge node {$\delta_3$} (m-1-3)
(m-1-3) edge node {$\cup e(\Lambda^3)$} (m-1-4)
(m-1-4) edge node {} (m-1-5)
(m-1-5) edge (m-1-6)

(m-2-1) edge (m-2-2)
(m-2-2) edge node {$\delta_2$} (m-2-3)
(m-2-3) edge node {$\cup e(\Lambda^2)$} (m-2-4)
(m-2-4) edge node {} (m-2-5)
(m-2-5) edge (m-2-6)

(m-1-2) edge (m-2-2)
(m-1-3) edge (m-2-3)
(m-1-4) edge (m-2-4)
(m-1-5) edge (m-2-5);
\end{tikzpicture}
\end{center}
By Proposition \ref{firstnontrivialiso},
$H^5(\Lambda^3,M;\IZ)\to H^5(\Lambda^2,M;\IZ)$ is an isomorphism; exactness of the triple sequences gives the corresponding isomorphism in equivariant degree $4$. Since $K_2=K_1^2$, Lemma \ref{lemmadegfundcorrespondence} gives
\[
H^6(\Lambda^2;\IZ)=0,
\qquad\text{hence}\qquad
H^6(\Lambda^2,M;\IZ)=0.
\]
Therefore, exactly as in the proof of Lemma \ref{lemequcohomgroups},
\[
\coker(\delta_3)\cong\coker(\delta_2)
\cong H^6_{\IS^1}(\Lambda^2,M;\IZ)
\cong\IZ\oplus\IZ/2\IZ.
\]
Moreover, the long exact sequence of the triple $(M,\Lambda^2,\Lambda^3)$ gives $H^6(\Lambda^3,M;\IZ)\cong\IZ$. Thus
\[
0\longrightarrow\coker(\delta_3)
\longrightarrow H^6_{\IS^1}(\Lambda^3,M;\IZ)
\longrightarrow H^6(\Lambda^3,M;\IZ)
\longrightarrow0
\]
splits, and
\[
H^6_{\IS^1}(\Lambda^3,M;\IZ)
\cong(\IZ\oplus\IZ/2\IZ)\oplus\IZ.
\]
Hence $H^6_{\IS^1}(\Lambda M,M;\IZ)$ contains $2$-torsion, contradicting Lemma \ref{lemequcohomgroups}.
\end{proof}


\newpage















\printbibliography


\end{document}